\documentclass[reqno]{amsart}
\usepackage{amssymb,amscd,graphicx,stmaryrd,mathrsfs,bbm,accents}
\usepackage[all]{xy}
\usepackage{color}
\usepackage{url}
\usepackage
{hyperref}

\newtheorem{theorem}{Theorem}[section]
\newtheorem{proposition}[theorem]{Proposition} 
\newtheorem{corollary}[theorem]{Corollary}
\newtheorem{lemma}[theorem]{Lemma}
\newtheorem{remark}[theorem]{Remark}

\newtheorem{remark&definition}[theorem]{Remark and Definition}

\makeatletter
  
  \@addtoreset{equation}{section}
\makeatother

\begin{document}

\title[Matrix Liberation Process III]{Matrix Liberation Process III: \\
\small Unitary Brownian motion and Martingale analysis} 
\author{Yoshimichi UEDA}
\address{
Graduate School of Mathematics, Nagoya University, 
Furocho, Chikusaku, Nagoya, 464-8602, Japan
}
\email{ueda@math.nagoya-u.ac.jp}
\date{\today}
\thanks{This work was supported by Grants-in-Aid for Scientific Research Grant Numbers JP18H01122, JP24K06757, JP25H00593(PI: Benoit Collins).}

\maketitle
\begin{abstract}
We investigate the rate functions that emerge in our previous works towards large deviation principle for the matrix liberation process driven by the unitary Brownian motion as well as the unitary Brownian motion itself. Our approach is grounded in the viewpoint of the martingale problem. Specifically, we formulate and solve a  ``free martingale problem" within this framework, which provides a new perspective on the underlying stochastic structure. 
\end{abstract}

\allowdisplaybreaks{


\section{Introduction}

This work is a sequel to our previous papers \cite{Ueda:JOTP19, Ueda:CJM21} on the matrix liberation process, which we introduced as a natural random matrix model of the \emph{liberation process} introduced by Voiculescu \cite{Voiculescu:AdvMath99}. Specifically, in the present paper, we investigate the rate functions that emerged in our prior studies toward a large deviation principle for the matrix liberation process as well as for the unitary Brownian motion, primarily from the perspective of martingale analysis. Furthermore, we establish a weak large deviation lower bound for these processes.

In Part VI \cite{Voiculescu:AdvMath99} of his series of free entropy, Voiculescu introduced and developed the \emph{free mutual information} $i^*$, which plays the r\^{o}le of mutual information in free probability. Defined for tuples of subalgebras in a tracial $W^*$-probability space, $i^*$ was constructed  via the liberation process in the same spirit as the microstate-free free entropy (see \cite{Voiculescu:InventMath98}). A decade later, the \emph{orbital free entropy} $\chi_\mathrm{orb}$ was introduced by Hiai, Miyamoto and us \cite{HiaiMiyamotoUeda:IJM09} under certain constraints, following the same spirit as microstates free entropy (see \cite{Voiculescu:InventMath94}), and was then generalized to full generality by us \cite{Ueda:IUMJ14} (see also \cite{BianeDabrowski:AdvMath13}). 

Since the negative orbital free entropy $-\chi_\mathrm{orb}$ shares many properties with the free mutual information $i^*$, we conjecture that the orbital free entropy is a microstates counterpart of free mutual information. This naturally leads to the unification problem on free mutual information, asking whether these two quantities are proportional. This question is the mutual information equivalent of the celebrated unification problem on free entropy (see e.g., \cite{Voiculescu:BLMS02}). 

The unification problem on free entropy was solved completely in the one-dimensional case by Voiculescu \cite{Voiculescu:InventMath98} using ``commutative analysis". In the multi-dimensional setting, it was partially solved as a byproduct of the large deviation upper bound established by Biane--Capitaine--Guionnet \cite{BianeCapitaineGuionnet:InventMath03} for matrix Brownian motion via ``free stochastic analysis". By contrast, the unification problem on free mutual information has thus far been solved only for the case of two projections (see \cite{HiaiUeda:AIHP09,CollinsKemp:JFA14,IzumiUeda:NJM15,Hamdi:NMJ20,Hamdi:CAOPT18}), results which are intimately connected to the free Jacobi process introduced by Demni \cite{Demni:JTOP08}.    

In this series of works, we have focused on the methodology of Biane--Capitaine--Guionnet \cite{BianeCapitaineGuionnet:InventMath03}, which is centered on the level 3 large deviation principle. Although an important consequence in their work regarding free entropy was recently reproved by Jekel--Pi \cite{JekelPi:DocMath24} using a more elementary alternative method, we maintain our focus on the approach of \cite{BianeCapitaineGuionnet:InventMath03}. Our objective is twofold: not only to address the unification problem on free mutual information but also to further advance the field of ``free stochastic analysis" as a byproduct. 

Perhaps the most significant, albeit technical, contribution of this paper is the solution to a kind of ``free martingale problem", which we briefly explain below. Consider the universal unital $C^*$-algebra with indeterminates $x_j = x_j^*$, $j \geq 1$, and $u_i(t), u_i(t)^*$, $i=1,\dots,n$, $t \geq 0$, such that    
$\Vert x_j\Vert_\infty \leq R$ and $u_i(t)^* u_i(t) = 1 = u_i(t)u_i(t)^*$. Let $\varphi$ be a ``continuous" tracial state on the $C^*$-algebra. 
We will work in the non-commutative $L^2$-space $L^2(\varphi)$ that is defined to be the separation/completion of the $C^*$-algebra with respect to the $L^2$-seminorm $a \mapsto \varphi(a^* a)^{1/2}$. 
Our primary objective is as follows. Suppose that the rate function $I=I_{\sigma_0}^\mathrm{uBM}$ that we constructed in section 7 of \cite{Ueda:CJM21} is finite at $\varphi$ such that 
\begin{equation}\label{rate_ft_expression}
I(\varphi) = \frac{1}{2}\int_0^\infty \Vert \xi(t)\Vert_{L^2(\varphi)^n}^2\,dt 
\end{equation}
with self-adjoint $\xi(t) := (\xi_1(t),\dots,\xi_n(t)) \in L^2(\varphi)^n$ adapted to the ``filtration" determined by $t \mapsto u(t)$. Then, we want to derive the $n$-dimensional free stochastic differential equation (free SDE) 
\begin{equation}\label{freeSDE}
du(t) = \sqrt{-1}(db(t)+\xi(t)\,dt)u(t) - \frac{1}{2}u(t)\,dt,
\end{equation} 
or more precisely, 
\[
du_i(t) = \sqrt{-1}\,db_i(t)\,u_i(t) + \left(\sqrt{-1}\,\xi_i(t) - \frac{1}{2}\right)u_i(t)\,dt, \quad i=1,\dots,n
\]
in $L^2(\varphi)$ with an appropriate $n$-dimensional free Brownian motion $b(t)=(b_1(t),\dots,b_n(t))$ adapted to the same filteration. This insight was suggested from Borell's formula for a diffusion process due to Lehec \cite{Lehec:Proceedings17}, and moreover, $I(\varphi)=0$ implies $\xi=0$, which says that the $u(t)$ must be an $n$-dimensional left increment free unitary Brownian motion that is known to be determined by free SDE \eqref{freeSDE} with $\xi=0$. In this paper, we will solve this problem under an additional assumption on $\varphi$ or more precisely on $\xi$ in such a way that $\xi$ is bounded over finite intervals with respect to the ``$L^\infty$-norm" rather than the $L^2$-norm. 

Assumption \eqref{rate_ft_expression} will be rephrased as a kind of ``non-commutative weak partial differential equation", from which one must derive free SDE \eqref{freeSDE}. The same procedure, of course, appeared in Biane--Capitaine--Guionnet's work \cite{BianeCapitaineGuionnet:InventMath03}, where they dealt with a simpler free SDE such as  $dx(t) = db(t) + \xi(t)\,dt$ and solved it under some regularity assumptions on $\xi$ by providing a free probability analogue of L\'{e}vy's characterization of Brownian motion. Since L\'{e}vy's characterization can be regarded as a solution to the simplest martingale problem, we regard our problem as a kind of ``free martingale problem" and solve it by following the idea of the martingale problem based on It\^{o} calculus.   

Solving our martingale problem requires developing a kind of stochastic calculus based on non-commutative martingales. Recently, such a general theory focusing on free probability has been developed by Jekel--Kemp--Nikitopoulos \cite{JekelKempNikitopoulos:JFA26}. Their work partially motivated the present study and clarified the nature of the problem; however, their theory is too broad to be applied directly to our concrete problem. Thus, we must develop a more ``specific theory" tailored to our particular non-commutative martingales. This is carried out in the present paper by building upon existing results from Pisier--Xu's seminal work \cite{PisierXu:CMP97}, as well as an important work by Dabrowski \cite{Dabrowski:preprint10} that refined the L{\'e}vy-type characterization of free Brownian motion mentioned earlier.     

With the solution to our free martingale problem explained above, we will prove, following the procedure in \cite{BianeCapitaineGuionnet:InventMath03}, a large deviation lower bound for the $n$-dimensional (left-increment) Brownian motion on the unitary group $\mathrm{U}(N)$ with speed $N^2$. By the standard contraction principle for the large deviation principle (see e.g., \cite{DemboZeitouni:Book}), this result also provides, in principle, a large deviation lower bound for the matrix liberation process as well. Furthermore, we establish the almost sure convergence of the empirical distribution of an $n$-dimensional process on the unitary group $\mathrm{U}(N)$, defined by a stochastic differential equation (SDE) in the usual sense, which is analogous to \eqref{freeSDE} with a special kind of drift. 

As mentioned above, we will mainly work within a framework that captures the (free) unitary Brownian motion rather than the (matrix) liberation process. To this end, we will recall the content of Section 7 of \cite{Ueda:CJM21} in some detail in Section \ref{unitary_process}. Section \ref{study_on_rate_function} provides an alternative representation of the rate function. In Section \ref{a_free_martingale_problem} we solve the free martingale problem mentioned above and apply it to Brownian motion on $\mathrm{U}(N)$, including a large deviation lower bound for it in Section \ref{applications_to_unitary_BM}. In Section \ref{matrix_liberation_process} we apply the results obtained to the matrix liberation process. Finally, Appendix \ref{noncomm_L^2} provides a brief explanation of non-commutative $L^2$-spaces and affiliated operators, using formulations and notations that fit ``free stochastic analysis".   

\section*{Acknowledgements} 
We would like to thank Yoann Dabrowski for fruitful discussions at Fields Institute in June 2019. We also benefited from a seminar talk by Evangelos Nikitopoulos at Kyoto University in May 2025 and subsequent conversations with him. We are grateful to him for sharing his insights on non-commutative stochastic calculus and would also like to thank Benoit Collins for organizing the seminar. That opportunity motivated me and enabled me to resume this research.

Finally, we would like thank Gemini (an AI language model by Google) for its assistance in refining the English expression. 

\section{Preliminaries on unitary processes}\label{unitary_process}

While we adhere to the notation of \cite[section 7]{Ueda:CJM21}, we introduce several modifications to streamline the presentation in this paper. 

\subsection{Noncommutative coordinates} 

Let $C_R^*\langle x,u\rangle$ be the universal $C^*$-algebra generated by indeterminates $x = (x_j)_{j \geq 1}$ and $u(t) = (u_1(t),\dots,u_n(t))$, $t \geq 0$, subject to the relations $x_j=x_j^*$ and $u_i(t)^ *u_i(t) = 1 = u_i(t)u_i(t)^*$, along with norm constraints $\Vert x_j\Vert \leq R$. This $C^*$-algebra is regarded as the space of ``continuous" test functions for an $n$-dimensional unitary process $u(t) = (u_1(t),\dots,u_n(t))$ together with countably many bounded self-adjoint random variables $x=(x_j)_{j\geq1}$. 

For simplicity, we have slightly modified the notation from the previous $C^*_R \langle x_{\diamond},u_\bullet(\,\cdot\,)\rangle$. 
Note that we have simplified the notation to $C_R^*\langle x,u\rangle$, dropping the symbols $\diamond$ and $\bullet$ used in \cite{Ueda:JOTP19, Ueda:CJM21} for improved readability. 

Let $\mathbb{C}\langle x,u\rangle$ be the universal $*$-algebra generated by $x=(x_j)_{j\geq1}$ and $u(t) = (u_1(t),\dots,u_n(t))$, $t \geq 0$, which is naturally and faithfully embedded in $C_R^*\langle x,u\rangle$. 

\subsection{Derivations and cyclic derivatives} 

For each $t \geq 0$ and $i = 1,\dots.n$, we introduce a unique derivation $\delta_{t,i} : \mathbb{C}\langle x,u\rangle \to \mathbb{C}\langle x,u\rangle\otimes \mathbb{C}\langle x,u\rangle$ determined by
\begin{equation*}
\begin{aligned} 
\delta_{t,i}u_{i'} (t') &
:= 
\delta_{i,i'} \mathbf{1}_{[0,t']}(t)\, (u_i(t')\otimes 1)(\sqrt{-1}\,u_i(t)^*\otimes u_i(t)), \\ 
\delta_{t,i}u_{i'} (t')^* &
:= 
\delta_{i,i'} \mathbf{1}_{[0,t']}(t)\, (-\sqrt{-1}\,u_i(t)^*\otimes u_i(t))(1\otimes u_i(t')^*), \\ 
\delta_{t,i} x_j 
&:= 0. 
\end{aligned}
\end{equation*}

Let $\theta$ be the flip-multiplication map from $\mathbb{C}\langle x,u\rangle\otimes \mathbb{C}\langle x,u\rangle$ to $\mathbb{C}\langle x,u\rangle$ defined by $\theta(a\otimes b) = ba$. We set $\mathfrak{D}_{t,i} := \theta\circ\delta_{t,i} : \mathbb{C}\langle x,u\rangle \to \mathbb{C}\langle x,u\rangle$. Note that $\delta_{t,i}$ and $\mathfrak{D}_{t,i}$ were previously denoted by $\delta_t^{(i)}$ and $\mathfrak{D}_t^{(i)}$, respectively. 

\subsection{Non-commutative $L^2$- and $L^\infty$-spaces}

Fix a tracial state $\varphi$ on $C^*_R\langle x,u\rangle$ for the time being. The $L^2$-space $L^2(\varphi)$ associated with $\varphi$ is available as mentioned in the introduction. 

By construction, we have a canonical linear map $C^*_R\langle x,u\rangle \to L^2(\varphi)$. Although this linear map may not be  injective, we still denote the image of $a \in C^*_R\langle x,u\rangle$ in $L^2(\varphi)$ by the same symbol $a$, following the usual convention in measure theory. 

Thanks to the trace property the adjoint operation $a \mapsto a^*$ on $C^*_R\langle x,u\rangle$ extends to $L^2(\varphi)$ isometrically. 
Furthermore, thanks to the trace property again, we can naturally endow $L^2(\varphi)$ with a bimodule structure over $C^*_R\langle x,u\rangle$. This bimodule structure induces an operation $\sharp$, which is a map from $(C_R^*\langle x,u\rangle\otimes C_R^*\langle x,u\rangle) \times L^2(\varphi)$ into $L^2(\varphi)$ extending:  
\[
(a\otimes b)\,\sharp\,\xi = a\xi b, \qquad a,b \in C^*_R\langle x,u\rangle,\ \xi \in L^2(\varphi), 
\]
where $C_R^*\langle x,u\rangle\otimes C_R^*\langle x,u\rangle$ denotes the algebraic tensor product. Note that $(a,b) \mapsto (a\otimes b)\,\sharp\,\xi$ is ``multiplicative" in the variable $a$ and ``anti-multiplicative" in the variable $b$. 

The tracial state $\varphi$ extends to $L^2(\varphi)$ in such a way that $\varphi(\xi) = (\xi|1)_{L^2(\varphi)}$ for all $\xi \in L^2(\varphi)$. 
For any $\xi \in L^2(\varphi)$, define the $L^\infty$-norm as:  
\[
\Vert \xi \Vert_{L^\infty(\varphi)} := 
\sup\{ \Vert \xi a\Vert_{L^2(\varphi)}; 
a \in C^*_R\langle x,u \rangle,\ 
\Vert a \Vert_{L^2(\varphi)} := \varphi(a^* a)^{1/2} = \varphi(aa^*)^{1/2} \leq 1 \}.
\] 
Although $\Vert\xi\Vert_{L^\infty(\varphi)}$ may be $+\infty$, the collection $L^\infty(\varphi)$ of all $\xi \in L^2(\varphi)$ with $\Vert \xi\Vert_{L^\infty(\varphi)} < +\infty$ contains the image of $C^*_R\langle x,u\rangle$. This image is dense in $L^2(\varphi)$, and $L^\infty(\varphi)$ itself forms a $W^*$-algebra (i.e., a space-free von Neumann algebra). Moreover, the $C^*$-norm $\Vert a \Vert$ of $a \in C^*_R\langle x,u\rangle$ is an upper bound for the $L^\infty$-norm; specifically, $\Vert a\Vert_{L^\infty(\varphi)} \leq \Vert a\Vert$ holds. 

For any $\xi \in L^\infty(\varphi)$, the trace property allows us to rigorously define the left-multiplication $\eta \mapsto \xi\eta$ and right-multiplication $\eta \mapsto \eta\xi$ as bounded operators on $L^2(\varphi)$, with their operator norms coinciding $\Vert\xi\Vert_{L^\infty(\varphi)}$. In the theory of operator algebras, using the concepts of left and right bounded vectors,  $L^\infty(\varphi)$ is identified with the von Neumann algebra generated by the left action $C^*_R\langle x,u\rangle$ on $L^2(\varphi)$. 

For each $t \geq 0$, let $\mathbb{C}\langle x,u\rangle_t$ be the unital $*$-subalgebra of $\mathbb{C}\langle x,u\rangle$ generated by $x_j$ and $u_i(s)$ for $s \leq t$, and let $C^*_R\langle x,u\rangle_t$ be its norm closure. We denote by $L^2(\varphi)_t$ the closure of the image of $C^*_R\langle x,u\rangle_t$ in $L^2(\varphi)$, and by $E_t^\varphi$ the orthogonal projection from $L^2(\varphi)$ onto $L^2(\varphi)_t$. 

Notably, $L^\infty(\varphi)_t := L^\infty(\varphi) \cap L^2(\varphi)_t$ is a $W^*$-subalgebra of $L^\infty(\varphi)$ and the $E_t^\varphi$ induces a unqiue $\varphi$-preserving (normal) conditional expectation from $L^\infty(\varphi)$ onto $L^\infty(\varphi)_t$. The \emph{filtration} determined by $t \mapsto u(t)$ under $\varphi$ refers to the increasing family of $W^*$-subalgebras $L^\infty(\varphi)_t $ (or equivalently, the increasing family of $L^2$-spaces $L^2(\varphi)_t$) for $t \geq 0$. 

While these concepts are standard within operator algebras, we provide a brief explanation involving affiliated operators in Appendix \ref{noncomm_L^2} for the reader's convenience.  

\subsection{Continuous tracial states}

A tracial state $\varphi$ on $C_R^*\langle x,u\rangle$ is said to be \emph{continuous}, if $t \mapsto u(t)$ is continuous on $L^2(\varphi)^n$ under left-multiplication in the strong operator topology. This is equivalent to the condition that $(t_1,\dots,t_m) \mapsto \varphi(w(t_1,\dots,t_m))$ is continuous for every $m \in \mathbb{N}$ and any word $w(t_1,\dots,t_m) = w_1(t_1)\cdots w_m(t_m)$, where each $w_k(t_k)$ is either $x_j$, $u_i(t_k)$, or $u_i(t_k)^*$ (see \cite[Lemma 2.1]{Ueda:JOTP19} for the proof, though a trivial modification is necessary). Let $TS^c(C_R^*\langle x,u\rangle)$ be the set of all continuous tracial states, which forms a complete metric space with the metric:  
\[
d(\varphi_1,\varphi_2) := 
\sum_{\ell = 1}^\infty\sum_{m = 1}^\infty \frac{1}{2^\ell (2R)^m}\max_{w \in \mathcal{W}_{\leq m}} \sup_{(t_1,\dots,t_m) \in [0,\ell]^m} |(\varphi_1-\varphi_2)(w(t_1,\dots,t_m))|,
\]
where $\mathcal{W}_{\leq m}$ denotes the set of all words as above with length not greater than $m$. See \cite[section 1, Remark on Part I]{Ueda:CJM21}, though a trivial modification is necessary here as well. 

A typical example of continuous tracial state is given by an $n$-dimensional Brownian motion $U_N(t) = (U_{N,1}(t),\dots, U_{N,n}(t))$ on the unitary group $\mathrm{U}(N)$ of rank $N$, together with $N\times N$ deterministic matrices $X_N = (X_{N,j})_{j\in\mathbb{N}}$ such that the operator norm $\Vert X_{N,j}\Vert \leq R$. This is constructed from an $n$-dimensional $N\times N$ self-adjoint matrix Brownian motion $H_N(t) = (H_{N,1}(t),\dots,H_{N,n}(t))$ as the unique solution to the stochastic differential equation (SDE) 
\begin{equation*} 
dU_{N,i}(t) = \sqrt{-1}\,dH_{N,i}(t) U_{N,i}(t) - \frac{1}{2}U_{N,i}(t)\,dt, \qquad U_{N,i}(0) = I.  
\end{equation*} 
See \cite[section 1]{Ueda:JOTP19}.  (Note that we use the symbol $U_{N,i}(t)$, $X_{N,j}$ and $H_{N,i}(t)$ instead of $U_N^{(i)}(t)$, $\xi_j(N)$ and $H_N^{(i)}(t)$, respectively, departing from the notation in the previous papers \cite{Ueda:JOTP19,Ueda:CJM21}.)  Specifically, we obtain a ``random" tracial state $\varphi_N = \varphi_{(X_N,U_N)} \in TS^c(C^*_R\langle x,u\rangle)$ from them as follows. Let $\pi_N = \pi_{(X_N,U_N)}$ be the unique ``random" unital $*$-homomorphism from $C^*_R\langle x,u\rangle$ into the algebra of $N\times N$ matrices $M_N(\mathbb{C})$ mapping each $x_j$ to $X_{N,j}$ and each $u_i(t)$ to $U_{N,i}(t)$. Then, the desired $\varphi_N$ is defined as the composition $\mathrm{tr}_N\circ\pi_N$, where $\mathrm{tr}_N := N^{-1}\mathrm{Tr}_N$ is the normalized trace on $M_N(\mathbb{C})$. 

Our desired large deviation principle concerns the sequence of random variables $\varphi_N$ taking values in $TS^c(C^*_R\langle x,u\rangle)$. Since a (continuous) tracial state is naturally regarded as the empirical probability distributions of a continuous unitary process, this large deviation principle should be regarded as being of level 2 in the usual sense. When $X_N$ has a limit $*$-joint distribution $\sigma_0$ as $N\to\infty$, we have already obtained, in \cite[section 7]{Ueda:CJM21}, a large deviation upper bound for the sequence of probability measures $\mathbb{P}(\varphi_N \in \,\cdot\,)$ with speed $N^2$ and the good rate function $I = I_{\sigma_0}^\mathrm{uBM} : TS^c(C^*_R\langle x,u\rangle) \to [0,+\infty]$ defined as follows. 

\subsection{Rate function}

We first add new indeterminates $\widetilde{u}(t) = (\widetilde{u}_1(t),\dots,\widetilde{u}_n(t))$ representing an $n$-dimensional unitary process (i.e., $\widetilde{u}_i(t)^*\widetilde{u}_i(t) = 1 = \widetilde{u}_i(t)\widetilde{u}_i(t)^*$ for all $t\geq0$) to $C^*_R\langle x,u\rangle \supset \mathbb{C}\langle x,u\rangle$,. This yields larger universal ($C^*$- and $*$-)algebras $C^*_R\langle x,u,\widetilde{u}\rangle \supset \mathbb{C}\langle x,u,\widetilde{u}\rangle$, which naturally contains $C^*_R\langle x,u\rangle \supset \mathbb{C}\langle x,u\rangle$ as subalgebras. (Note that in \cite{Ueda:CJM21}, we used the notation $v(t) = (v_1(t),\dots,v_n(t))$ in place of the current $\widetilde{u}(t) = (\widetilde{u}_1(t),\dots,\widetilde{u}_n(t))$.)  

For each $t \geq 0$ there exists a unique unital $*$-homomorphism $\Pi^t : C^*_R\langle x,u\rangle \to C^*_R\langle x,u,\widetilde{u}\rangle$ mapping each $u_i(t')$ to: 
\begin{equation*}
u_i^t(t') := \widetilde{u}_i((t'-t)\vee 0)u_i(t\wedge t').
\end{equation*}  

Now, we fix $\varphi \in TS^c(C^*_R\langle x,u\rangle)$, and then define the desired rate function $I(\varphi)$. 

By taking the reduced free product for unital $C^*$-algebras (see e.g.\ \cite[Chapter 8]{FieldsInstMono13}), we extend the $\varphi$ to a unique tracial state $\widetilde{\varphi}$ on $C^*_R\langle x,u,\widetilde{u}\rangle$ such that:  
\begin{itemize}   
\item[(e.i)] The restriction of $\widetilde{\varphi}$ to $C^*_R\langle x,u\rangle$ is exactly $\varphi$.  
\item[(e.ii)] The distribution of $\widetilde{u}(t) = (\widetilde{u}_1(t),\dots,\widetilde{u}_n(t))$ under $\widetilde{\varphi}$ is that of an $n$-dimensional, left-increment free unitary Brownian motion (i.e., a freely independent $n$-tuple of left-increment, free unitary Brownian motions). 
\item[(e.iii)] The process $\widetilde{u}(t)$ is $*$-freely independent of $C^*_R\langle x,u\rangle$ under $\widetilde{\varphi}$.   
\end{itemize}
We set $\varphi^t := \widetilde{\varphi}\circ\Pi^t$, which clearly belongs to $TS^c(C^*_R\langle x,u\rangle)$ for every $t \geq 0$. This $\varphi^t$ plays the r\^{o}le of the empirical probability distribution of the unitary process distributed under $\varphi$ up to time $t$, and  following the law of a left increment free unitary Brownian motion after time $t$. Similarly, by taking another reduced free product, we construct $\sigma_0^\mathrm{frBM} \in TS^c(C^*_R\langle x,u\rangle)$ such that:  
\begin{itemize} 
\item[(f.i)] The joint distribution of $x=(x_j)_{j\geq1}$ is $\sigma_0$, 
\item[(f.ii)] The distribution of $u(t) = (u_1(t),\dots,u_n(t))$ under $\widetilde{\varphi}$ is that of an $n$-dimensional, left-increment free unitary Brownian motion.  
\item[(f.iii)] The variables $x_j$ and the process $u(t)$ are $*$-freely independent under $\sigma_0^{\mathrm{frBM}}$. 
\end{itemize} 
By the left-increment property, we have $(\sigma_0^\mathrm{frBM})^t = \sigma_0^\mathrm{frBM}$ for all $t \geq 0$. Using these notations, we define:  
\[
I(\varphi;a,T) := \varphi^T(a) - \sigma_0^\mathrm{frBM}(a) - \frac{1}{2} \int_0^T \sum_{i=1}^n \Vert E^{\varphi^t}_t(\mathfrak{D}_{t,i}a)\Vert_{L^2(\varphi^t)}^2\,dt
\]
for any $a = a^* \in \mathbb{C}\langle x,u\rangle$ and $T \geq 0$ (see Corollary \ref{C3.3} for its well-definedness), and then:  
\[
I(\varphi) := \sup\big\{ I(\varphi;a,T);\ T > 0,\ a = a^* \in \mathbb{C}\langle x,u\rangle\big\}. 
\]
If $\varphi$ does not agree with $\sigma_0$ on the polynomials in $x_j$, then $I(\varphi)$ must be $+\infty$ (see e.g., \cite[Lemma 5.3]{Ueda:JOTP19}).  

\section{Study on Rate Function}\label{study_on_rate_function} 

The primary objective of this section is to derive a specific expression of $\varphi \in TS^c(C^*_R\langle x, u\rangle)$ with $I(\varphi) < +\infty$. To achieve this, we will reformulate the rate function $I(\varphi)$ into a more tractable expression. 

\begin{lemma}\label{L3.1} 
For any word $w$ in the variables $x_j$ and $u_i(t), u_i(t)^*$, we have:  
\[
\mathfrak{D}_{t,i}w = \sqrt{-1}\left(\sum_{\substack{w = w_1 u_i(t') w_2 \\ t \leq t'}} u_i(t) w_2 w_1 u_i(t')u_i(t)^* - \sum_{\substack{w = w_1 u_i(t')^* w_2 \\ t \leq t'}} u_i(t) u_i(t')^* w_2 w_1 u_i(t)^*\right)
\]
and 
\[
\Pi^t(\mathfrak{D}_{t,i}w) = \sqrt{-1}\left(\sum_{\substack{w = w_1 u_i(t') w_2 \\ t \leq t'}} u_i(t) \Pi^t(w_2 w_1) v_i(t'-t) - \sum_{\substack{w = w_1 u_i(t')^* w_2 \\ t \leq t'}} v_i(t'-t)^* \Pi^t(w_2 w_1) u_i(t)^*\right).
\]
If $w$ does not contain any letters $u_i(t'), u_i(t')^*$ with $t' \geq t$, then $\mathfrak{D}_{t,i}w = 0$.  
\end{lemma}
\begin{proof}
This lemma follows from a straightforward, albeit tedious, calculation. We leave the details to the reader. 
\end{proof}

Let $\varphi \in TS^c(C^*_R\langle x,u\rangle)$ be arbitrarily given, and consider its extension $\widetilde{\varphi}$ to $C_R^*\langle x,u,\widetilde{u}\rangle$ as described in Subsection 2.5. By replacing $(C^*_R\langle x, u\rangle,\varphi)$ with $(C^*_R\langle x,u,\widetilde{u}\rangle,\widetilde{\varphi})$, we obtain the $L^2$- and the $L^\infty$-spaces $L^2(\widetilde{\varphi}) \supseteq  L^\infty(\widetilde{\varphi})$ as in Subsection 2.3. We have the following natural inclusions: 
\[
\vcenter{
\xymatrix{
L^2(\varphi)\ \ar@{^{(}->}[r] \ar@{}[d]|{\bigcup} &  L^2(\widetilde{\varphi}) \ar@{}[d]|{\bigcup} \\
L^\infty(\varphi)\ \ar@{^{(}->}[r] & L^\infty(\widetilde{\varphi}) \ar@{}[lu]|{\circlearrowleft}\\
C^*_R\langle x,u\rangle \ar[u] \ar@{}[r]|*{\quad \subset\quad\ } & C^*_R\langle x,u,\widetilde{u}\rangle, \ar[u] \ar@{}[lu]|{\circlearrowleft} 
}
}
\]
where two horizontal arrows represent natural embeddings with closed range, and two bottom vertical arrows denote  the canonical linear maps (which are not necessarily injective). Let $E_\infty^{\widetilde{\varphi}}$ be the orthogonal projection from $L^2(\widetilde{\varphi})$ onto $L^2(\varphi)$ ($= L^\infty(\varphi)_\infty$), which induces a unique $\widetilde{\varphi}$-preserving (normal) conditional expectation from $L^\infty(\widetilde{\varphi})$ onto $L^\infty(\varphi)$. 

\begin{lemma}\label{L3.2} 
For every $a \in \mathbb{C}\langle x,u\rangle$ and each $i=1,\dots,n$, the following identity holds: 
\[
E_t^{\varphi^t}(\mathfrak{D}_{t,i}a) = E_\infty^{\widetilde{\varphi}}(\Pi^t(\mathfrak{D}_{t,i}a)).
\] 
\end{lemma}
\begin{proof}
Since the restriction of $\varphi^t$ to $C^*_R\langle x,u\rangle_t$ coincides with the original $\varphi$, we naturally have $L^2(\varphi^t)_t = L^2(\varphi)_t \subset L^2(\varphi)$. In particular, $E_t^{\varphi^t}(\mathfrak{D}_{t,i}a)$ belongs to $L^2(\widetilde{\varphi})$. 

Owing to the natural inclusions above, it suffices to prove that 
\[
(E_t^{\varphi^t}(\mathfrak{D}_{t,i}a)|b)_{L^2(\widetilde{\varphi})} 
=
(E_\infty^{\widetilde{\varphi}}(\Pi^t(\mathfrak{D}_{t,i}a))|b) _{L^2(\widetilde{\varphi})} 
\]
for all $b\in C^*_R\langle x,u\rangle_t$. Since $b \in L^2(\varphi^t)_t = L^2(\varphi)_t \subset L^2(\varphi) \subset L^2(\widetilde{\varphi})$, the left-hand side reduces to
\[
(E_t^{\varphi^t}(\mathfrak{D}_{t,i}a)|b)_{L^2(\varphi^t)} 
= 
(\mathfrak{D}_{t,i}a|b)_{L^2(\varphi^t)} 
=
\varphi^t((\mathfrak{D}_{t,i}a)b^*). 
\]
Conversely, the right-hand side becomes 
\[
(\Pi^t(\mathfrak{D}_{t,i}a)|b)_{L^2(\widetilde{\varphi})}\, 
\big(= \widetilde{\varphi}(\Pi^t(\mathfrak{D}_{t,i}a)b^*)\big) 
= 
\widetilde{\varphi}(\Pi^t((\mathfrak{D}_{t,i}a)b^*)) 
= 
\varphi^t((\mathfrak{D}_{t,i}a)b^*),
\]
where we used $\Pi^t(b^*) = b^*$ by definition. Hence we are done. 
\end{proof}

\begin{corollary}\label{C3.3} 
For every $a \in \mathbb{C}\langle x,u\rangle$ and each $i=1,\dots,n$, the map $t \mapsto E_t^{\varphi^t}(\mathfrak{D}_{t,i}a) = E_\infty^{\widetilde{\varphi}}(\Pi^t(\mathfrak{D}_{t,i}a))$ is a compactly supported, bounded, and piecewise strongly continuous function taking values in $L^\infty(\varphi)$. Consequently, the norm function 
\[
t \mapsto \Vert E_t^{\varphi^t}(\mathfrak{D}_{t,i}a)\Vert_{L^2(\varphi^t)} = \Vert E_\infty^{\widetilde{\varphi}}(\Pi^t(\mathfrak{D}_{t,i}a))\Vert_{L^2(\varphi)}
\]
is bounded, piecewise continuous. Furthermore, $E_t^{\varphi^t}(\mathfrak{D}_{t,i}a) = E_\infty^{\widetilde{\varphi}}(\Pi^t(\mathfrak{D}_{t,i}a))$ is contained in $\mathbb{C}\langle x,u\rangle_t$ for every $t \geq 0$
\end{corollary} 
\begin{proof} 
By Lemma \ref{L3.1}, $\Pi^t(\mathfrak{D}_{t,i}a)$ is piecewise (in $t$) a polynomial in $x_j$ and $u^t_i(t'), u^t_i(t')^*$. Consequently, the first assertion follows from that the fact that $t \mapsto u_i^t(t') = \widetilde{u}_i((t'-t)\vee0)u_i(t\wedge t')$ is strongly continuous on $L^2(\widetilde{\varphi})$ for a fixed $t'$ and that $E_\infty^{\widetilde{\varphi}}$ is strongly continuous on bounded balls in $L^\infty(\varphi)$. 

The second assertion can be verified by applying item (e.ii) from the definition of $\widetilde{\varphi}$, following the idea of \cite[Theorem 19 and Exercises 19,20 in section 2.5]{MingoSpeicher:Book} (cf.\ \cite[Lemmas 4.3, 4.4]{Ueda:JOTP19}). 
\end{proof} 

Let $\xi : [0,+\infty) \to L^2(\varphi)^n$ be a function given by $\xi(t) = (\xi_1(t),\dots,\xi_n(t))$ such that each component $\xi_i(t)$ belongs to $L^2(\varphi)_t$ and is self-adjoint. A primary example is the map $t \mapsto (E_\infty^{\widetilde{\varphi}}(\Pi^t(\mathfrak{D}_{t,i}a)))_{i=1}^n$ for $a=a^* \in \mathbb{C}\langle x,u\rangle$ as seen in Corollary \ref{C3.3}. These examples form a dense linear submanifold of the $L^2$-space of such functions; see Lemma \ref{L3.10} below. 

Since the restriction of $\varphi^t$ to $C^*_R\langle x,u\rangle_t$ coincides with the original $\varphi$, we naturally have $L^2(\varphi^t)_t = L^2(\varphi)_t$. Based on this remark, we introduce a self-adjoint derivation $D_{\xi(t)}: \mathbb{C}\langle x,u\rangle \to L^2(\varphi^t)$ in such a way that 
\begin{align*}
D_{\xi(t)}u_i(t') 
&= 
\mathbf{1}_{[0,t']}(t)\,u_i(t')(\sqrt{-1}u_i(t)^*\xi_i(t) u_i(t)), \\
D_{\xi(t)} x_j &= 0. 
\end{align*}
Here, a derivation $D$ is said to be self-adjoint if it commutes with the adjoint operation, i.e., $D(a^*) = D(a)^*$ for all $a$. 

\begin{lemma}\label{L3.4} 
For every $a \in \mathbb{C}\langle x,u\rangle$, we have the following identity for the derivation $D_{\xi(t)}$: 
\[
D_{\xi(t)}a = \sum_{i=1}^n (\delta_{t,i}\,a)\,\sharp\,\xi_i(t),
\]
which implies:  
\[
\sum_{i=1}^n \varphi^t(E_t^{\varphi^t}(\mathfrak{D}_{t,i}a)\xi_i(t)) 
= 
\sum_{i=1}^n (E_t^{\varphi^t}(\mathfrak{D}_{t,i}a)|\xi_i(t))_{L^2(\varphi^t)} 
= 
\varphi^t(D_{\xi(t)}a).  
\]
\end{lemma}

This lemma establishes that the family $D_\xi = (D_{\xi(t)})_{t\geq0}$ plays the r\^{o}le of the ``gradient in the direction of $\xi$". 

\begin{proof}
It is straightforward to see that the map $a \mapsto \sum_{i=1}^n (\delta_{t,i}\,a)\,\sharp\,\xi_i(t)$ defines a derivation that is self-adjoint, provided $\xi_i(t) = \xi_i(t)^*$. Hence, it suffices to verify the first identity for the generators $x_j$ and $u_{i'}(t')$. Indeed, we have $\delta_{t,i}x_j = 0$ and 
\[
\sum_{i=1}^n \delta_{t,i} u_{i'}(t')\,\sharp\,\xi_i(t) 
= 
\mathbf{1}_{[0,t']}(t)\,u_{i'}(t')(\sqrt{-1} u_{i'}(t)^* \xi_{i'}(t) u_{i'}(t)), 
\]
which confirms the first identity. 

The first equality of the second identity follows directly from the self-adjointness $\xi_i(t) = \xi_i(t)^*$. Thus, it remains only to prove the second equality. Using the Sweedler notation $\delta_{t,i}a = a(i)_{(1)}\otimes a(i)_{(2)}$ and noting that $\xi_i(t) \in L^2(\varphi^t)_t$, we have
\begin{align*}
\sum_{i=1}^n (E_t^{\varphi^t}(\mathfrak{D}_{t,i}a)|\xi_i(t))_{L^2(\varphi^t)} 
&=
\sum_{i=1}^n (a(i)_{(2)}a(i)_{(1)}|\xi_i(t))_{L^2(\varphi^t)} \\
&= 
\sum_{i=1}^n (1|a(i)_{(2)}^*\xi_i(t)a(i)_{(1)}^*)_{L^2(\varphi^t)} \\
&= 
\sum_{i=1}^n (a(i)_{(1)}\xi_i(t)a(i)_{(2)}|1)_{L^2(\varphi^t)} \\
&= 
\sum_{i=1}^n ((\delta_{t,i}a)\,\sharp\,\xi_i(t)|1)_{L^2(\varphi^t)} 
= 
\varphi^t(D_{\xi(t)}a).  
\end{align*}
Hence we are done.
\end{proof} 

We introduce new coordinates $y(t) = (y_1(t),\dots,y_n(t))$ and $\widetilde{y}(t) = (\widetilde{y}_1(t),\dots,\widetilde{y}_n(t))$ defined by: 
\[
y_i(t) := e^{t/2}u_i(t) \in \mathbb{C}\langle x,u\rangle \qquad \text{and} \qquad \widetilde{y}_i(t) := e^{t/2}\widetilde{u}_i(t) \in \mathbb{C}\langle x,u,\widetilde{u}\rangle.
\]
The Introduction of these coordinates is inspired by the discussions in \cite {Biane:Fields97}. Clearly, the variables $x_j$ and $y_i(t)$, $y_i(t)^{-1}$ altogether generate $\mathbb{C}\langle x,u\rangle$ algebraically (though not as $*$-algebra); that is, any element of $\mathbb{C}\langle x,u\rangle$ is a polynomial in $x_j$ and $y_i(t)$, $y_i(t)^{-1} = e^{-t/2}u_i(t)^*$. Similarly, $x$ and $y(t)$, $\widetilde{y}(t)$, $y(t)^{-1}$, $\widetilde{y}(t)^{-1}$ algebraically generate $\mathbb{C}\langle x,u,\widetilde{u} \rangle$. 

It is necessary to examine the behavior of the previously defined maps under these new coordinates. The proof of the next lemma
is straightforward and is thus left to the reader.  

\begin{lemma}\label{L3.5} 
The following identities hold: 
\begin{itemize}
\item[(1)] $D_{\xi(t)}y_i(t') = \mathbf{1}_{[0,t']}(t)\,y_i(t')(\sqrt{-1}\,y_i(t)^{-1}\xi_i(t)y_i(t))$ for every $i=1,\dots,n$ and each $t' \geq 0$.  
\item[(2)] $\Pi^t(y_i(t')) = \widetilde{y}_i((t'-t)\vee 0)y_i(t\wedge t')\, (=: y_i^t(t'))$ for every $i=1,\dots,n$ and each $t' \geq 0$. 
\end{itemize}
\end{lemma} 

The next lemma asserts that $\Pi^t$ extends to $L^2(\varphi^t)$ and interacts consistently with the derivation $D_{\xi(t)}$. 

\begin{lemma}\label{L3.6} 
The $*$-homomorphism $\Pi^t$ extends to $L^2(\varphi^t)$ as an isometric linear map into $L^2(\widetilde{\varphi})$. This extension, still denoted by the same symbol, satisfies:
\[
\Pi^t(c\,\sharp\,\eta) = (\Pi^t\otimes\Pi^t)(c)\,\sharp\,\eta
\]
for any $c \in C^*_R\langle x,u\rangle\otimes C_R^*\langle x,u\rangle$ and $\eta \in L^2(\varphi)_t = L^2(\varphi^t)_t$.
In particular, the identity $\varphi^t(D_{\xi(t)}a) = \widetilde{\varphi}(\Pi^t(D_{\xi(t)}a))$ holds for every $a \in \mathbb{C}\langle x,u\rangle$. 
\end{lemma}
\begin{proof}
The existence of the desired extension is clear from the definition of $\varphi^t$. For any $\eta \in L^2(\varphi^t)$, we can choose a sequence $b_k \in C^*_R\langle x,u\rangle_t$ such that $\Vert b_k - \eta\Vert_{L^2(\varphi^t)_t} \to 0$ as $k\to\infty$. Since $L^2(\varphi^t)_t = L^2(\varphi)_t \subseteq L^2(\varphi) \subset L^2(\widetilde{\varphi})$, we have $\Vert b_k - \eta\Vert_{L^2(\widetilde{\varphi})} \to 0$ as $k\to\infty$. Given that $\Pi^t(b_k) = b_k$ for all $k$, it follows that    
\begin{align*}
\Pi^t(c\,\sharp\,\eta) = \lim_{k\to\infty} \Pi^t(c\,\sharp\,b_k) = \lim_{k\to\infty} (\Pi^t\otimes\Pi^t)(c)\,\sharp\,b_k = (\Pi^t\otimes\Pi^t)(c)\,\sharp\,\eta. 
\end{align*}
The final assertion follows immediately from Lemma \ref{L3.4}.
\end{proof}

\begin{remark}\label{R3.7}
Lemma \ref{L3.6} provides a clear procedure for calculating $\varphi^t(D_{\xi(t)}a)$ for any $a\in \mathbb{C}\langle x,u\rangle$. First, express $a$ as a polynomial in the variables $x_j$, $y_i(t')$, $y_i(t')^{-1}$. Next, compute $D_{\xi(t)}a$ using the identities in Lemma \ref{L3.5}(1), combined with the Leibniz rule and the self-adjointness of the derivation. Then, substitute $y_i(t')$ with $y_i^t(t')$ as per Lemma \ref{L3.6}(2), and finally evaluate the expression under the tracial state $\widetilde{\varphi}$.
\end{remark}  

The following is a key explicit calculation utilized in our framework. 

\begin{lemma}\label{L3.8} For every $s > r \geq 0$, every $a \in \mathbb{C}\langle x,u\rangle_r$ and all $i,i'=1,\dots,n$, the following identity holds:  
\[
E_\infty^{\widetilde{\varphi}}(\Pi^t(\mathfrak{D}_{t,i}(y_{i'}(s)-y_{i'}(r))a)) = \delta_{i,i'}\,\sqrt{-1}\,\mathbf{1}_{(r,s]}(t) y_i(t)a, \qquad t \geq 0.
\] 
\end{lemma}
\begin{proof} 
This result follows from a direct, albeit tedious, calculation using Lemma \ref{L3.1}. We leave the details to the interested reader. 
\end{proof} 

We define several Hilbert spaces as follows. Let $\mathcal{H}(\varphi)_0$ be all the bounded, compactly supported and piecewise continuous functions $\xi : [0,+\infty) \to L^2(\varphi)^n$ such that $\xi(t) \in (L^2(\varphi)_t)^n$ for every $t \geq 0$. We denote by $\mathcal{H}(\varphi)$ the closure of $\mathcal{H}(\varphi)_0$ in the Hilbert space $L^2([0,+\infty);L^2(\varphi)^n)$. Furthermore, for each $T>0$, we consider the closed subspace $\mathcal{H}(\varphi)_T := \mathbf{1}_{[0,T)}\cdot\mathcal{H}(\varphi)$, consisting of functions supported on the interval $[0,T)$. 

The adjoint operation on $L^2(\varphi)$ naturally induces an adjoint on $L^2([0,+\infty);L^2(\varphi)^n)$ defined by $\xi_i^*(t) := \xi_i(t)^*$ for $\xi(t) = (\xi_1(t),\dots,\xi_n(t)) \in L^2([0,+\infty); L^2(\varphi)^n)$. Clearly, this operation is isometric with respect to the $L^2$-norm. 

The next lemma is standard, and thus its proof is omitted. 

\begin{lemma}\label{L3.9}
The following properties hold: 
\begin{itemize}
\item[(1)] Any element of $\mathcal{H}(\varphi)$ admits a representative function $\xi : [0,+\infty) \to L^2(\varphi)^n$ such that $\xi(t) \in (L^2(\varphi)_t)^n$ for all $t \geq 0$. Consequently, $\mathcal{H}(\varphi)$ is closed under the adjoint operation. 
\item[(2)] The collection of functions of the form $(\mathbf{1}_{(s_i,t_i]}\,a_i)_{i=1}^n$, where $t_i > s_i \geq 0$ and $a_i \in \mathbb{C}\langle x,u\rangle_{s_i}$, is total in $\mathcal{H}(\varphi)$. 
\end{itemize}  
\end{lemma}

\begin{lemma}\label{L3.10} 
For every $T>0$, the collection of functions given by $t\mapsto (E_\infty^{\widetilde{\varphi}}(\Pi^t(\mathfrak{D}_{t,i}a)))_{i=1}^n$, with $a \in \mathbb{C}\langle x,u\rangle_s$ and $s < T$, constitudes a dense linear submanifold of $\mathcal{H}(\varphi)_T$. Moreover, the subcollection satisfying the constraint $a = a^*$ forms a dense linear submanifold (over the real numbers $\mathbb{R}$) of the real Hilbert space $\mathcal{H}(\varphi)^\mathrm{sa}_T$, consisting of all the self-adjoint elements of $\mathcal{H}(\varphi)_T$. 
\end{lemma}
\begin{proof} 
The second part follows from the first part, because $E_\infty^{\widetilde{\varphi}}$, $\Pi^t$ and $\mathfrak{D}_{t,i}$ are all $*$-preserving. 

By Lemma \ref{L3.9}(2), it suffices to prove that any $t\mapsto (\mathbf{1}_{(s_i,t_i]}(t) a_i)_{i=1}^n$ with $t_i > s_i \geq 0$ and $a_i \in \mathbb{C}\langle x,u\rangle_{s_i}$ belongs to the closure of the collection. 

For any $t_i \geq r_{i,2} > r_{i,1} \geq s_i$, we consider 
\[
a := -\sqrt{-1}\,\sum_{i=1}^n (y_i(r_{i,2})-y_i(r_{i,1})) y_i(r_{i,1})^{-1} a_i.
\]
By Lemma \ref{L3.8}, we have 
\[
E_\infty^{\widetilde{\varphi}}(\Pi^t(\mathfrak{D}_{t,i}(a))) 
= 
\mathbf{1}_{(r_{i,1},r_{i,2}]}(t)y_i(t)y_i(r_{i,1})^{-1} a_i 
= 
\mathbf{1}_{(r_{i,1},r_{i,2}]}(t)\,e^{(t-r_{i,1})/2}u_i(t)u_i(r_{i,1})^* a_i. 
\] 
Hence we obtain that 
\begin{align*}
&\int_0^{+\infty} \sum_{i=1}^n \Vert E_\infty^{\widetilde{\varphi}}(\Pi^t(\mathfrak{D}_{t,i}a)) - \mathbf{1}_{(r_{i,1},r_{i,2}]}(t)a_i\Vert_{L^2(\varphi)}^2\,dt \\
&=
\sum_{i=1}^n \int_{r_{i,1}}^{r_{i,2}} \Vert (e^{(t-r_{i,1})/2} u_i(t)u_i(r_{i,1})^* - 1)a_i\Vert_{L^2(\varphi)}^2\,dt \\
&\leq 
\sum_{i=1}^n \Vert a_i\Vert^2\, (r_{i,2}-r_{i,1})\,\big((e^{(r_{i,2}-r_{i,1})/2}-1) + \max_{t \in [r_{i,1},r_{i,2}]}\Vert u_i(t) - u_i(r_{i,1}) \Vert_{L^2(\varphi)}^2\big).
\end{align*}
Since $t \mapsto u_i(t)$ is continuous in $L^2(\varphi)$ (and hence uniformly continuous on any finite intervals), we conclude the desired assertion by dividing each interval $(s_i,t_i]$ into small subintervals.  
\end{proof}

The same proof as \cite[Lemma 5.3]{Ueda:JOTP19} shows the following: 

\begin{lemma}\label{L3.11} 
If $I(\varphi) < +\infty$, then 
\[
I(\varphi) 
= 
\frac{1}{2}\sup_{\substack{T > 0 \\ a = a^* \in \mathbb{C}\langle x,u\rangle}} 
\frac{(\varphi^T(a) - \sigma_0^\mathrm{frBM}(a))^2}{\int_0^T \sum_{i=1}^n \Vert E_\infty^{\widetilde{\varphi}}(\Pi^t(\mathfrak{D}_{t,i}a))\Vert_{L^2(\varphi)}^2\,dt},
\]
where we adopt the convention $0/0 = 0$. In particular, for all $T > 0$ and any self-adjoint $a = a^* \in \mathbb{C}\langle x, u\rangle$, we have the inequality
\[
(\varphi^T(a) - \sigma_0^\mathrm{frBM}(a))^2 
\leq 
2 I(\varphi) 
\int_0^T \sum_{i=1}^n \Vert E_\infty^{\widetilde{\varphi}}(\Pi^t(\mathfrak{D}_{t,i}a))\Vert_{L^2(\varphi)}^2\,dt. 
\]
The right-hand side can be expressed more concisely as
\[
(\varphi^T(a) - \sigma_0^\mathrm{frBM}(a))^2 
\leq 
2I(\varphi) \big\Vert t \mapsto (\mathbf{1}_{[0,T)}(t)\, E_\infty^{\widetilde{\varphi}}(\Pi^t(\mathfrak{D}_{t,i}a)))_{i=1}^n\big\Vert_{L^2}^2,  
\]
where $\Vert\,\cdot\,\Vert_{L^2}$ denotes the $L^2$-norm for functions taking its values in $L^2(\varphi)^n$. 
\end{lemma}

Applying the above lemma together with the Riesz representation theorem for Hilbert spaces, we observe that for each $T > 0$, there exists a unique $\xi_T \in \mathcal{H}(\varphi)_T^\mathrm{sa}$ such that 
\begin{equation}\label{Eq3.1}
\begin{aligned}
\varphi^T(a) - \sigma_0^\mathrm{frBM}(a) 
&= 
\big(t \mapsto (\mathbf{1}_{[0,T)}(t)\, E_\infty^{\widetilde{\varphi}}(\Pi^t(\mathfrak{D}_{t,i}a)))_{i=1}^n\,\big|\,\xi_T\big)_{\mathcal{H}(\varphi)_T^\mathrm{sa}}^2 \\
&= 
\int_0^T \sum_{i=1}^n \varphi(E_\infty^{\widetilde{\varphi}}(\Pi^t(\mathfrak{D}_{t,i}a))\xi_{T,i}(t))\,dt \\
&= 
\int_0^T \sum_{i=1}^n \varphi^t(E_\infty^{\widetilde{\varphi}}(\Pi^t(\mathfrak{D}_{t,i}a))\xi_{T,i}(t))\,dt \\
&= 
\int_0^T \varphi^t(D_{\xi_T(t)}a)\,dt \qquad \text{(by Lemma \ref{L3.4})} 
\end{aligned}
\end{equation}
for all self-adjoint $a = a^* \in \mathbb{C}\langle x,u\rangle$. 

Now, consider an arbitrary pair $T_2 > T_1 > 0$. For any $a = a^* \in \mathbb{C}\langle x,u\rangle_T$ with $T < T_1$, we have 
\begin{align*} 
\int_0^{T_1} \sum_{i=1}^n \big(E_\infty^{\widetilde{\varphi}}(\Pi^t(\mathfrak{D}_{t,i}a))|\xi_{T_1,i}(t)-\xi_{T_2,i}(t))_{L^2(\varphi)}\,dt 
&= 
(\varphi^{T_2}(a) - \sigma_0^\mathrm{frBM}(a))-(\varphi^{T_1}(a) - \sigma_0^\mathrm{frBM}(a)) \\
&= 
(\varphi(a) - \sigma_0^\mathrm{frBM}(a))-(\varphi(a) - \sigma_0^\mathrm{frBM}(a)) 
= 0
\end{align*}
since $\mathfrak{D}_{t,i}a = 0$ if $t \geq T_1 (> T)$. By Lemma \ref{L3.10} it follows that 
\[
\int_0^{T_1} \Vert \xi_{T_2,i}(t) - \xi_{T_1,i}(t)\Vert_{L^2(\varphi)}^2\,dt = 0, \qquad i=1,\dots,n, 
\]
implying that $\xi_{T_1,i}(t) = \xi_{T_2,i}(t)$ for a.e.\ $t \in [0,T_1)$ and all $i=1,\dots,n$. 

Consequently, by involing Lemma \ref{L3.9}(1), we can identify a unique self-adjoint function $\xi : [0,+\infty) \to L^2(\varphi)^n$ such that $\xi(t) \in (L^2(\varphi)_t)^n$ for every $t \geq 0$, and for which the restriction $\mathbf{1}_{[0,T)}\,\xi$ represents the element $\xi_T \in (\mathcal{H}(\varphi)_T^\mathrm{sa})^n$ for each $T > 0$. 

This leads to the main theorem of this section, which serves as the unitary Brownian motion counterpart to \cite[Theorem 5.2(2)]{BianeCapitaineGuionnet:InventMath03}. 

\begin{theorem}\label{T3.12} 
Assume that $I(\varphi) < +\infty$. Then there exists a unique (up to almost everywhere equivalence) self-adjoint function $\xi : [0,+\infty) \to L^2(\varphi)^n$ such that $\xi(t) \in (L^2(\varphi)_t)^n$ for every $t \geq 0$, satisfying the following properties: 
\begin{itemize}
\item[(i)] $\displaystyle I(\varphi) = \frac{1}{2} \Vert\xi\Vert_{L^2}^2$, where $\Vert\,\cdot\,\Vert_{L^2}$ denotes the $L^2$-norm for functions taking its values in $L^2(\varphi)^n$. In particular, $\xi$ belongs to the Hilbert space $\mathcal{H}(\varphi)$.  
\item[(ii)] For every $T > 0$ and any $a \in \mathbb{C}\langle x,u\rangle$, the following identity holds: 
\[
\varphi^T(a) = \sigma_0^\mathrm{frBM}(a) + \int_0^T \varphi^t(D_{\xi(t)}a)\,dt.
\]  
\end{itemize}

Moreover, if there exists a square-integrable, self-adjoint function $\xi : [0,+\infty) \to L^2(\varphi)^n$ such that $\xi(t) \in (L^2(\varphi)_t)^n$ for all $t \geq 0$ and satisfying the integral expression in item (ii) above, then item (i) above necessarily holds with this $\xi$.   
\end{theorem} 
\begin{proof} 
Following the preceding considerations, we observe that for any $a=a^* \in \mathbb{C}\langle x,u\rangle$, the functional $I(\varphi; a,T)$  can be expressed as
\begin{align*}
I(\varphi;a,T) 
&=
\varphi^T(a) - \sigma_0^\mathrm{frBM}(a) - \frac{1}{2}\big\Vert t \mapsto (\mathbf{1}_{[0,T)}(t)\, E_\infty^{\widetilde{\varphi}}(\Pi^t(\mathfrak{D}_{t,i}a)))_{i=1}^n\big\Vert_{L^2}^2 \\
&= 
\big(t \mapsto (\mathbf{1}_{[0,T)}(t)\, E_\infty^{\widetilde{\varphi}}(\Pi^t(\mathfrak{D}_{t,i}a)))_{i=1}^n\,\big|\,\mathbf{1}_{[0,T)}\,\xi\big)_{L^2} - \frac{1}{2}\big\Vert t \mapsto (\mathbf{1}_{[0,T)}(t)\, E_\infty^{\widetilde{\varphi}}(\Pi^t(\mathfrak{D}_{t,i}a)))_{i=1}^n\big\Vert_{L^2}^2 \\ 
&= 
\frac{1}{2}\Vert\mathbf{1}_{[0,T)}\,\xi\Vert_{L^2}^2 - \frac{1}{2}\big\Vert \big(t \mapsto (\mathbf{1}_{[0,T)}(t)\, E_\infty^{\widetilde{\varphi}}(\Pi^t(\mathfrak{D}_{t,i}a)))_{i=1}^n\big) - \mathbf{1}_{[0,T)}\,\xi\big\Vert_{L^2}^2. 
\end{align*} 
By Lemma \ref{L3.10}, the second term can be made arbitrarily small by choosing an appropriate $a$. Consequently, we conclude that
\[
\sup_{a = a^* \in \mathbb{C}\langle x,u\rangle} I(\varphi;a,T) = \frac{1}{2}\Vert\mathbf{1}_{[0,T)}\,\xi\Vert_{L^2}^2,   
\]
from which item (i) immediately follows. Item (ii) was established in the discussion preceding the theorem.

For the second part of the theorem, the same algebraic observation holds. If item (ii) is satisfied, we can perform the calculation in \eqref{Eq3.1} in reverse to justify the identity in item (i). Hence we are done.  
\end{proof} 

\begin{remark}\label{R3.13}{\rm
Roughly speaking, item (ii) in Theorem \ref{T3.12} implies that the family $\varphi^t$ is a solution to the following initial-value problem for  a ``non-commutative partial differential equation" of weak form: 
\[
\partial_t \langle a,\varphi^t\rangle = \langle D_{\xi(t)} a, \varphi^t\rangle, \qquad a \in  \mathbb{C}\langle x,u\rangle,\ t > 0
\]
subject to the initial condition $\varphi^0 = \sigma_0^\mathrm{frBM}$. Here, $\langle\,\cdot\,,\,\cdot\,\rangle$ denotes the dual pairing between $C^*_R\langle x,u\rangle$ and $TS^c(C^*_R\langle x,u\rangle)$.
}
\end{remark} 

\section{A free martingale problem}\label{a_free_martingale_problem}

Throughout this section, we assume that a tracial state $\varphi \in TS^c(C^*_R\langle x,u\rangle)$ and a measurable function $\xi : [0,+\infty) \to L^\infty(\varphi)^n$ ($\subset L^2(\varphi)^n$), satisfying $\xi(t) \in (L^\infty(\varphi)_t)^n$ for every $t \geq 0$, fulfill the following conditions: 
\begin{itemize} 
\item[(A)] Local Boundedness: $\xi$ is bounded over any finite intervals in the $L^\infty$-norm, 
\item[(B)] Integral Representation: Item (ii) in Theorem \ref{T3.12} holds; that is, by Lemma \ref{L3.6} one has
\begin{equation}\label{Eq4.1} 
\varphi^T(a) = \sigma_0^\mathrm{frBM}(a) + \int_0^T \widetilde{\varphi}(\Pi^t(D_{\xi(t)}a))\,dt 
\end{equation} 
for every $T > 0$ and any $a \in \mathbb{C}\langle x,u\rangle$.  
\end{itemize}
In what follows, we will repeatedly use the formula \ref{Eq4.1}, typically, choosing a sufficiently large $T$.   
 
We start with the following lemma: 
 
\begin{lemma}\label{L4.1} 
For each $t \geq 0$ we define $z(t) = (z_1(t),\dots,z_n(t)) \in L^\infty(\varphi)^n$ by 
\[
z_i(t) := y_i(t) - \sqrt{-1}\int_0^t \xi_i(s)\,y_i(s)\,ds, \qquad i=1,\dots,n.
\]
Then, every $z_i(t)$ is a continuous martingale with respect to the filtration $(L^\infty(\varphi)_t)_{t \geq 0}$; that is, $E_r^\varphi(z_i(s)) = z_i(r)$ holds for any $s \geq r \geq 0$. 
\end{lemma}
\begin{proof}
Fix an arbitrary $a \in \mathbb{C}\langle x,u\rangle_r$. By combining Lemma \ref{L3.8} with Lemma \ref{L3.2} we obtain that 
\[
E_t^{\varphi^t}(\mathfrak{D}_{t,i'}((y_{i}(s)-y_{i}(r))a^*)) = \delta_{i,i'}\,\sqrt{-1}\,\mathbf{1}_{(r,s]}(t) y_i(t)a^* \qquad i' = 1,\dots,n.
\] 
By formula \eqref{Eq4.1} in conjunction with Lemmas \ref{L3.6}, \ref{L3.4}, we have 
\begin{align*}
\varphi((y_i(s)-y_i(r))a^*) 
&= \varphi^s((y_i(s)-y_i(r))a^*) \\
&= \sigma_0^\mathrm{frBM}((y_i(s)-y_i(r))a^*) + \sqrt{-1}\int_r^s \varphi^t(y_i(t) a^* \xi_i(t))\,dt \\
&= \sigma_0^\mathrm{frBM}((y_i(s)-y_i(r))a^*) - \sqrt{-1}\int_r^s \varphi^t(\xi_i(t)y_i(t) a^*)\,dt,
\end{align*}
where we used the trace property. Since $y_i(t) = e^{t/2}u_i(t)$ behaves a martingale with respect to the filtration determined by $t \mapsto u(t)$ under $\sigma_0^\mathrm{frBM}$, i.e., $(L^\infty(\sigma_0^\mathrm{frBM})_t)_{t \geq 0}$,  by \cite[Subsection 2.3]{Biane:Fields97}, the first term on the last line must vanish, implying that $(z_i(s) - z_i(r)|a)_{L^2(\varphi)}  = 0$. This directly leads to the desired assertion.  
\end{proof} 

\begin{remark} \label{R4.2}
The aforementioned lemma remains valid even when $\xi \in \mathcal{H}(\varphi)$; however, in this more general case, we can only guarantee that the resulting martingale belongs to $L^2(\varphi)$.  
\end{remark}

Using this martingale $z(t)=(z_1(t),\dots,z_n(t))$, we construct the non-commutative stochastic integral 
\begin{equation}\label{Eq4.2}
b_i(t) := -\sqrt{-1}\int_0^t dz_i(s)\,y_i(s)^{-1}, \qquad i=1,\dots,n,
\end{equation}
which defines a martingale adapted to the same filtration $(L^2(\varphi)_t)_{t \geq 0}$. This construction can be carried out by essentially following standard techniques in stochastic analysis (see, e.g., \cite[section 3.2]{KaratzasShreve:Book}), as previously demonstrated by Jekel--Kemp--Nikitopoulos \cite[section 4]{JekelKempNikitopoulos:JFA26}. 

We will prove that $b(t) = (b_1(t),\dots,b_n(t))$ is an $n$-dimensional free Brownian motion adapted to the filtration $(L^\infty(\varphi)_t)_{t \geq 0}$, utilizing Dabrowski's generalization \cite[Theorem 22]{Dabrowski:preprint10} of Biane--Capitaine--Guionnet's L\'{e}vy-type characterization \cite[Theorem 6.2]{BianeCapitaineGuionnet:InventMath03}.  

\begin{remark}[Heuristic derivation of the formula \eqref{Eq4.2}] \label{R4.3}
The following provides the intuition behind the derivation of the formula \eqref{Eq4.2}. If the formula \eqref{freeSDE} holds, the dynamics are given by $dy(t) = \sqrt{-1}(db(t)+\xi(t)\,dt)\,y(t)$. This suggests that the driving Brownian motion can be formally expressed as 
\[
db(t) = -\sqrt{-1}(dy(t)-\sqrt{-1}\,\xi(t)y(t)\,dt)y(t)^{-1}. 
\]
To formalize this using stochastic integration, we first established in Lemma \ref{L4.1} that the process $dy(t) - \sqrt{-1}\,\xi(t)y(t)\,dt$  yields a martingale.   
\end{remark} 

In what follows, we will repeatedly use the following properties:  

\begin{lemma}\label{L4.4}
$\widetilde{\varphi}(\widetilde{y}_i(t)) = 1$ and $\widetilde{\varphi}(\widetilde{y}_i(t)^2) = 1-t$ for every $t \geq 0$. 
\end{lemma} 
\begin{proof}
Since $\widetilde{y}_i(t) = e^{t/2}\widetilde{u}_i(t)$ and since $\widetilde{u}_i(t)$ is a left-increment unitary Brownian motion under $\widetilde{\varphi}$, the desired facts immediately follow from \cite[Lemma 1]{Biane:Fields97}.    
\end{proof}

The next lemma is a key in our discussion. Although its proof involves a straightforward but tedious calculation based on the formula \eqref{Eq4.1}, the result is essential for our construction. 

\begin{lemma}\label{L4.5} 
For any $s > r \geq 0$, $i,j = 1,\dots,n$ and $a, b \in \mathbb{C}\langle x,u\rangle_r$, the following identity holds:  
\[
\varphi((z_i(s)-z_i(r))a(z_j(s)-z_j(r))^* b) 
=
\delta_{i,j}(e^s - e^r)\varphi(a)\varphi(b) = \delta_{i,j}\int_r^s \varphi(a)\varphi(b)\,e^t\,dt. 
\]
In particular, the measure $\kappa_{z_i}$ of the martingale $z_i(t)$ in the sense of \cite[Lemma 4.13]{JekelKempNikitopoulos:JFA26} is precisely $\kappa_{z_i}(dt) = e^t\,dt$.  
\end{lemma}
\begin{proof}
We first observe that 
\begin{align*}
\varphi((z_i(s)-z_i(r))a(z_j(s)-z_j(r))^* b) 
&=
\varphi(y_i(s)-y_i(r))a(y_j(s)-y_j(r))^* b) \\
&\quad-
\sqrt{-1}\int_r^s \varphi(\xi_i(t)y_i(t) a (y_j(s)-y_j(r))^* b)\,dt \\
&\quad+ 
\sqrt{-1}\int_r^s \varphi((y_i(s)-y_i(r))a y_j(t)^* \xi_j(t) b)\,dt \\
&\quad+ 
\int_r^s \int_r^s \varphi(\xi_i(t_1)y_i(t_1) a y_j(t_2)^* \xi_j(t_2) b)\,dt_1\,dt_2. 
\end{align*} 
We will compute these four terms individually by applying the formula \eqref{Eq4.1}. 

Applying the formula \eqref{Eq4.1} to the first term in conjunction with Remark \ref{R3.7}, the free independence (e.iii) in Subsection 2.5, Lemma \ref{L4.4} and $\widetilde{y}_i(t)\widetilde{y}_i(t)^* = e^t 1$, we have 
\begin{align*}
&\varphi(y_i(s)-y_i(r))a(y_j(s)-y_j(r))^* b) \\
&=
\widetilde{\varphi}((\widetilde{y}_i(s-r)-1)(y_i(r)a y_j(r)^*)(\widetilde{y}_j(s-r)-1)^* b) \\ 
&\quad+ \sqrt{-1} 
\int_r^s 
\Big\{ 
\widetilde{\varphi}(\widetilde{y}_i(s-t)(\xi_i(t)y_i(t)ay_j(t)^*)\widetilde{y}_j(s-t)^* b) - 
\widetilde{\varphi}(\widetilde{y}_i(s-t)\xi_i(t)y_i(t)ay_j(r)^* b)
\\
&\phantom{aaaaaaaaaai}
-  \widetilde{\varphi}((\widetilde{y}_i(s-t)(y_i(t)ay_j(t)^* \xi_j(t))\widetilde{y}_i(s-t)^* b) 
+ \widetilde{\varphi}(y_i(r)ay_j(t)^* \xi_j(t)\widetilde{y}_j(s-t)^* b)
\Big\}\,dt  \\
&= 
\widetilde{\varphi}((\widetilde{y}_i(s-r)-1)(\widetilde{y}_j(s-r)-1)^*)\varphi(y_i(r)ay_j(r)^*)\varphi(b) \\
&\quad+ \sqrt{-1} 
\int_r^s 
\Big\{ 
\widetilde{\varphi}(\widetilde{y}_i(s-t)\widetilde{y}_j(s-t)^*) \varphi(\xi_i(t)y_i(t)ay_j(t)^*)\varphi(b) 
+ \varphi(\xi_i(t)y_i(t)ay_j(t)^* b) \\
&\phantom{aaaaaaaaaaaaaaaaaaaaaaaaaaaaaaa} 
-\varphi(\xi_i(t)y_i(t)ay_j(t)^*)\varphi(b)
- \varphi(\xi_i(t)y_i(t)ay_j(r)^* b)
\\
&\phantom{aaaaaaaaaai}
- \widetilde{\varphi}(\widetilde{y}_i(s-t)\widetilde{y}_j(s-t)^*)\varphi(y_i(t)ay_j(t)^* \xi_j(t))\varphi(b) 
- \varphi(y_i(t)ay_j(t)^* \xi_j(t)b) \\
&\phantom{aaaaaaaaaaaaaaaaaaaaaaaaaaaaaaa}
+ \varphi(y_i(t)ay_j(t)^* \xi_j(t))\varphi(b) 
+ \varphi(y_i(r)ay_j(t)^* \xi_j(t)b)
\Big\}\,dt \\
&= 
\delta_{i,j}(e^s-e^r)\varphi(a)\varphi(b) \\
&\quad+ \sqrt{-1} 
\int_r^s 
\Big\{ 
\varphi(\xi_i(t)y_i(t)ay_j(t)^* b)
- \varphi(\xi_i(t)y_i(t)ay_j(r)^* b) \\
&\phantom{aaaaaaaaaai} 
- \varphi(y_i(t)ay_j(t)^* \xi_j(t)b)
+ \varphi(y_i(r)ay_j(t)^* \xi_j(t)b)
\Big\}\,dt.
\end{align*}
Using the same facts that we used above, we also have
\begin{align*}
&\sqrt{-1}\int_r^s \varphi(\xi_i(t)y_i(t) a (y_j(s)-y_j(r))^* b)\,dt \\
&= 
\sqrt{-1}\int_r^s 
\Bigg\{
\widetilde{\varphi}(\xi_i(t)y_i(t)a(\widetilde{y}_j(s-t)y_j(t)-y_j(r))^* b) \\
&\phantom{aaaaaaaaaai}
-\sqrt{-1} \int_t^s \widetilde{\varphi}(\xi_i(t)y_i(t)ay_j(t')^* \xi_j(t')\widetilde{y}_j(s-t')^*b)\,dt'
\Bigg\}dt' \\
&= 
\sqrt{-1}\int_r^s \Big\{ \varphi(\xi_i(t)y_i(t)ay_j(t)^* b) - \varphi(\xi_i(t)y_i(t)ay_j(r)^* b)\Big\}\,dt \\
&\phantom{aa}
+\int\int_{r \leq t_1 \leq t_2 \leq s} \varphi(\xi_i(t_1)y_i(t_1)a y_j(t_2)^* \xi_j(t_2)b)\,dt_1 dt_2. 
\end{align*} 
Similarly, 
\begin{align*} 
&\sqrt{-1}\int_r^s \varphi((y_i(s)-y_i(r))a y_j(t)^* \xi_j(t) b)\,dt \\
&= 
\sqrt{-1}\int_r^s \Big\{ \varphi(y_i(t)ay_j(t)^* \xi_j(t)b) - \varphi(y_i(r)ay_j(r)^* \xi_j(t)b)\Big\}\,dt \\
&\phantom{aa}
+\int\int_{r \leq t_2 \leq t_1 \leq s} \varphi(\xi_i(t_1)y_i(t_1)a y_j(t_2)^* \xi_j(t_2)b)\,dt_1 dt_2. 
\end{align*}
These calculations immediately imply the desired formula.
\end{proof} 

As mentioned previously, we adopt the standard framework for constructing stochastic integrals within an $L^2$-setting; we refer the reader to \cite[Subsection 4.2]{JekelKempNikitopoulos:JFA26} for its specific adaptation to the non-commutative context. Specifically, the stochastic integral appearing in the formula \eqref{Eq4.2} is well defined in $L^2(\varphi)$ as the limit of an approximating net 
\begin{equation*} 
b_{i,\triangle}(t) := -\sqrt{-1}\sum_{k=1}^m (z_i(s_k)-z_i(s_{k-1}))y_i(s_{k-1})^{-1} 
\end{equation*}
over all possible divisions $\triangle : 0 = s_0 < s_1 < \dots < s_m = t$, and 
\begin{equation}\label{Eq4.3}
\Vert b_i(s) - b_i(r)\Vert_{L^2(\varphi)}^2 = \left\Vert \int_r^s dz_i(dt)\,y_i(t)^{-1}\right\Vert_{L^2(\varphi)}^2 = \int_r^s \Vert y_i(t)^{-1}\Vert_{L^2(\varphi)}^2\,\kappa_{z_i}(dt) = \int_r^s e^{-t}\,e^t\,dt = s-r 
\end{equation} 
for any $s > r \geq 0$. 
Consequently, the definition of $b_i(t)$ is well defined and its covariance coincides with that of free Brownian motion. We will next prove that the $b_i(t)$ is self-adjoint. 

\begin{lemma}\label{L4.6} 
$b_i(t) = b_i(t)^*$ for every $t \geq 0$ and $i=1,\dots,n$.  
\end{lemma}
\begin{proof}
Observe that 
\begin{align*}
&\Vert(-\sqrt{-1}(z_i(s_k)-z_i(s_{k-1}))y_i(s_{k-1})^{-1}) - (-\sqrt{-1}(z_i(s_k) - z_i(s_{k-1}))y_i(s_{k-1})^{-1})^*\Vert_{L^2(\varphi)}^2 \\
&=  
\Vert (z_i(s_k)-z_i(s_{k-1}))y_i(s_{k-1})^{-1} + (y_i(s_{k-1})^{-1})^*(z_i(s_k)-z_i(s_{k-1}))^*\Vert_{L^2(\varphi)}^2 \\
&=
2\Bigg(\varphi((z_i(s_k)-z_i(s_{k-1}))y_i(s_{k-1})^{-1}(y_i(s_{k-1})^{-1})^*(z_i(s_k)-z_i(s_{k-1}))^*) \\
&\phantom{aaaaaaaaaa}
+ \mathrm{Re}\varphi((z_i(s_k)-z_i(s_{k-1}))y_i(s_{k-1})^{-1}(z_i(s_k)-z_i(s_{k-1}))y_i(s_{k-1})^{-1})\Bigg) \\
&=  
2\Bigg( \int_{s_{k-1}}^{s_k} \Vert y_i(s_{k-1})^{-1}\Vert_{L^2(\varphi)}^2\,e^t\,dt 
+ \mathrm{Re}\Big(-(t-s) + \sqrt{-1}\int_{s_{k-1}}^{s_k} 2(s_k -t)\varphi(\xi_i(t))\,dt\Big)\Bigg) \\
&= 
2\Bigg( \int_{s_{k-1}}^{s_k} \Vert y_i(s_{k-1})^{-1}\Vert_{L^2(\varphi)}^2\,e^t\,dt -(s_k - s_{k-1})\Bigg). 
\end{align*} 
By the martingale property (Lemma \ref{L4.1}) we obtain that 
\begin{align*}
&\Vert b_i(t) - b_i(t)^*\Vert_{L^2(\varphi)}^2 
= 
\lim_{|\triangle| \to 0} \Vert b_{i,\triangle}(t) - b_{i,\triangle}(t)^*\Vert_{L^2(\varphi)}^2 \\
&= 
\lim_{|\triangle| \to 0} \sum_{k=1}^m \Vert(-\sqrt{-1}(z_i(s_k)-z_i(s_{k-1}))y_i(s_{k-1})^{-1}) - (-\sqrt{-1}(z_i(s_k) - z_i(s_{k-1}))y_i(s_{k-1})^{-1})^*\Vert_{L^2(\varphi)}^2 \\
&= 
2\lim_{|\triangle| \to 0} \sum_{k=1}^m \Bigg( \int_{s_{k-1}}^{s_k} \Vert y_i(s_{k-1})^{-1}\Vert_{L^2(\varphi)}^2\,e^t\,dt -(s_k - s_{k-1})\Bigg) \\
&= 
2 \Big(\int_0^t e^{-t}\,e^t\,dt - t\Big) = 0, 
\end{align*} 
where $|\triangle| := \max\{ |s_k - s_{k-1}|; k=1,\dots,m\}$. Hence we are done. 
\end{proof}

\begin{lemma}\label{L4.7} For any $s > r \geq 0$, $i,j=1,\dots,n$ and $a, b \in \mathbb{C}\langle x,u\rangle_r$, the following identity holds: 
\[
\varphi((b_i(s)-b_i(r))a(b_j(s)-b_j(r))b) = \delta_{i,j} (s-r)\varphi(a)\varphi(b) + o(s-r) 
\]
as $s-r \searrow 0$
\end{lemma}
\begin{proof}
Observe that
\begin{align*}
&\Vert (b_i(s) - b_i(r)) - (-\sqrt{-1}(z_i(s)-z_i(r))y_i(r)^{-1}\Vert_{L^2(\varphi)}^2 \\
&\quad= 
\int_r^s \Vert y_i(t)^{-1} - y_i(r)^{-1}\Vert_{L^2(\varphi)}^2\,\kappa_{z_i}(dt) \\
&\quad\leq 
\sup_{r \leq t \leq s}\Vert e^{-t/2}u_i(t)^* - e^{-r/2}u_i(r)^*\Vert_{L^2(\varphi)}^2\,(e^s - e^r).
\end{align*}
Hence, using (the last part of) Lemma \ref{L4.5} we have
\begin{align*}
&|\varphi((b_i(s)-b_i(r))a(b_j(s)-b_j(r))b) \\
&\phantom{aaaaaaaaaaaa}- 
\varphi((z_i(s)-z_i(r))y_i(r)^{-1}a (y_j(r)^{-1})^*(z_j(s)-z_j(r))^* b))| \\
&= 
|\varphi((b_i(s)-b_i(r))a(b_j(s)-b_j(r))^* b) \\
&\phantom{aaaaaaaaaaaa}- 
\varphi((-\sqrt{-1}(z_i(s)-z_i(r))y_i(r)^{-1})a((-\sqrt{-1}(z_j(s)-z_j(r))y_j(r)^{-1})^* b))| \\
&\leq 
\Vert a \Vert\,\Vert b\Vert\,\Vert (b_i(s)-b_i(r)) - (-\sqrt{-1}(z_i(s)-z_i(r))y_i(r)^{-1})\Vert_{L^2(\varphi)} \Vert b_j(s)-b_j(r)\Vert_{L^2(\varphi)} \\
&\quad+ 
\Vert a \Vert\,\Vert b\Vert\,\Vert (z_i(s)-z_i(r))\Vert_{L^2(\varphi)} \Vert (b_j(s)-b_j(r)) - (-\sqrt{-1}(z_j(s)-z_j(r))y_j(r)^{-1})\Vert_{L^2(\varphi)} \\
&=
\Vert a\Vert\,\Vert b\Vert\, \sup_{r \leq t \leq s} \Vert e^{-t/2}u_i(t) - e^{-r/2}u_i(r)\Vert_{L^2(\varphi)}\sqrt{(e^s-e^r)(s-r)} \\
&\quad+ 
\Vert a\Vert\,\Vert b\Vert\,\sqrt{e^s-e^r}\sup_{r \leq t \leq s} \Vert e^{-t/2}u_j(t) - e^{-r/2}u_j(r)\Vert_{L^2(\varphi)}\sqrt{e^s-e^r} \\
&= o(s-r) 
\end{align*}
as $s-r \searrow 0$, since $\varphi$ is a continuous tracial state. Consequently, we obtain that 
\begin{align*}
&\varphi((b_i(s)- b_i(r))a(b_j(s)-b_j(r))b) \\
&= 
\varphi((z_i(s)-z_i(r))y_i(r)^{-1}a (y_j(r)^{-1})^*(z_j(s)-z_j(r))^* b)) + o(s-r) \\
&= 
\delta_{i,j} (e^s - r^r) \varphi(y_i(r)^{-1}a (y_i(r)^{-1})^*)\varphi(b) + o(s-r) \\
&= 
\delta_{i,j} (s-r)\varphi(a)\varphi(b) + o(s-r)
\end{align*}
as $s-r \searrow \infty$ by Lemma \ref{L4.5} again. 
\end{proof}   

It is necessary to consider the processes $b_i(t)$ within the $L^4$-space $L^4(\varphi)$. For the reader's convenience, we provide a practical characterization of $L^4(\varphi)$. For any $\xi \in L^2(\varphi)$, we define the $L^4$-norm as  
\[
\Vert \xi \Vert_{L^4(\varphi)} := \sup\{ \Vert \xi a\Vert_{L^2(\varphi)};\ a \in C^*_R\langle x,u\rangle,\ \varphi(|a|^4) \leq 1\},  
\]
which may take the value  $+\infty$. We then denote by $L^4(\varphi)$ the set of all $\xi \in L^2(\varphi)$ such that $\Vert \xi\Vert_{L^4(\varphi)} < +\infty$. It is easy to verify that $\Vert\xi\Vert_{L^2(\varphi)} \leq \Vert\xi\Vert_{L^4(\varphi)}$ for every $\xi \in L^4(\varphi)$, and that $L^4(\varphi)$ is a Banach space with this norm. Moreover, it is clear, from the definition of the $L^4$-norm, that $L^4(\varphi)$ is closed under the left and right multiplications of elements of $L^\infty(\varphi)$. While the generalized H\"{o}lder inequality underlies this definition, it is not explicitly required here. Indeed, by applying the trace property alongside the Cauchy--Schwarz inequality, one can easily show that $\varphi(|\xi\eta|^2) \leq \varphi(|\xi|^4)^{1/2}\varphi(|\eta|^4)^{1/2}$ for all $\xi,\eta \in L^\infty(\varphi)$. From this, it follows that $L^\infty(\varphi) \subseteq L^4(\varphi)$ and $\Vert \xi\Vert_{L^4(\varphi)} = \varphi(|\xi|^4)^{1/4}$ for all $\xi \in L^\infty(\varphi)$. Moreover, a standard approximation argument shows that $\Vert\xi\eta\Vert_{L^2(\varphi)} \leq \Vert\xi\Vert_{L^4(\varphi)} \Vert\eta\Vert_{L^4(\varphi)}$ holds for any pair $\xi,\eta \in L^4(\varphi)$. 

\begin{lemma}\label{L4.8} 
For each $T > 0$, there exists a constant $C_{1,T} > 0$ such that $\Vert z_i(s) - z_i(r)\Vert_{L^4(\varphi)}^4 \leq C_{1,T}(s-r)^2$ holds for any $0 \leq r < s \leq T$ with $0 \leq s-r \leq 1/16$. 
\end{lemma}
\begin{proof} 
By the formula \eqref{Eq4.1} in conjunction with Remark \ref{R3.7}, the trace property, the free independence (e.iii) in Subsection 2.5, Lemma \ref{L4.4} and $\widetilde{y}_i(t)\widetilde{y}_i(t)^* = e^t 1$, we have
\begin{align*}
\Vert y_i(s)-y_i(r)\Vert_{L^2(\varphi)}^2 
&= 
\widetilde{\varphi}((\widetilde{y}_i(s-r)-1)y_i(r)y_i(r)^*(\widetilde{y}_i(s-r)-1)^*) \\
&\quad+ 
\int_r^s \Big\{
\widetilde{\varphi}((\sqrt{-1}\widetilde{y}_i(s-t)\xi_i(t)y_i(t))(\widetilde{y}_i(s-t)y_i(t)-y_i(r))^*) \\
&\phantom{aaaaaaaaaa}+ 
\widetilde{\varphi}((\widetilde{y}_i(s-t)y_i(t)-y_i(r))(\sqrt{-1}\widetilde{y}_i(s-t)\xi_i(t)y_i(t))^*)\Big\}\,dt \\
&= 
e^r \widetilde{\varphi}(|\widetilde{y}_i(s-r)-1|^2) 
-2\int_r^s \mathrm{Im}\,\varphi(y_i(r)y_i(t)^*\xi_i(t))\,dt \\
&\leq 
e^r \Vert \widetilde{y}_i(s-r)-1 \Vert_{L^\infty(\widetilde{\varphi})}^2 + 4\sup_{r \leq t \leq s} \Vert\xi_i(t)\Vert_{L^\infty(\varphi)}\, e^r (e^{(s-r)/2}-1) \\
&\leq 
e^r\left(\frac{2\sqrt{s-r}}{1-2\sqrt{s-r}}\right)^2 + 4\sup_{r \leq t \leq s} \Vert\xi_i(t)\Vert_{L^\infty(\varphi)}\, e^r (e^{(s-r)/2}-1)
\end{align*}
as long as $s-r < 1/4$ by the proof of \cite[Lemma 8]{Biane:Fields97}. Consequently, letting 
\[
C := 8\,e^T\left(1 + 2\sup_{0 \leq t \leq T}\Vert\xi_i(t)\Vert_{L^\infty(\varphi)}\,(e^{1/32}-1)\right)
\]
one has 
\[
\Vert y_i(s)-y_i(r)\Vert_{L^2(\varphi)}^2 
\leq 2C (s-r)
\]
for all $0 \leq r \leq s \leq T$ with $s-r \leq 1/16$. 

Similarly, we have 
\begin{align*} 
&\Vert y_i(s)-y_i(r)\Vert_{L^4(\varphi)}^4 \\
&= 
\widetilde{\varphi}((\widetilde{y}_i(s-r)-1)y_i(r)y_i(r)^*(\widetilde{y}_i(s-r)-1)^* (\widetilde{y}_i(s-r)-1)y_i(r)y_i(r)^*(\widetilde{y}_i(s-r)-1)^*) +
2\sqrt{-1} \times\\
&\times \int_r^s \Big\{\widetilde{\varphi}((\widetilde{y}_i(s-t)\xi_i(t)y_i(t))(\widetilde{y}_i(s-t)y_i(t)-y_i(r))^* (\widetilde{y}_i(s-t)y_i(t)-y_i(r))(\widetilde{y}_i(s-t)y_i(t)-y_i(r))^*) \\
&\quad- 
\widetilde{\varphi}((\widetilde{y}_i(s-t)y_i(t)-y_i(r))(\widetilde{y}_i(s-t)y_i(t)-y_i(r))^*(\widetilde{y}_i(s-t)y_i(t)-y_i(r))(\widetilde{y}_i(s-t)\xi_i(t)y_i(t))^*)\Big\}\,dt \\
&= 
e^{2r}\widetilde{\varphi}(|\widetilde{y}_i(s-r)-1|^4) \\
&\quad-
4\int_r^s \mathrm{Im}\widetilde{\varphi}(\widetilde{y}_i(s-t)\xi_i(t)y_i(t)(\widetilde{y}_i(s-t)y_i(t)-y_i(r))^* \\
&\phantom{aaaaaaaaaaaaaaaaaaaaaa}
 (\widetilde{y}_i(s-t)y_i(t)-y_i(r))(\widetilde{y}_i(s-t)y_i(t)-y_i(r))^*)\,dt \\
&\leq 
e^{2r}\Vert \widetilde{y}_i(s-r)-1\Vert_{L^\infty(\widetilde{\varphi})}^4 \\
&\quad+ 
4\int_r^s \Vert (\widetilde{y}_i(s-t)\xi_i(t)y_i(t)(\widetilde{y}_i(s-t)y_i(t)-y_i(r))^*\Vert_{L^\infty(\widetilde{\varphi})}\,\Vert \widetilde{y}_i(s-t)y_i(t)-y_i(r)\Vert_{L^2(\widetilde{\varphi})}^2\,dt \\
&\leq 
16^2\,e^{2T}\,(s-r)^2 + 16\,e^T\,\sup_{0 \leq t \leq T}\Vert\xi_i(t)\Vert_{L^\infty(\varphi)} 
\int_r^s \Big(e^t\,\Vert \widetilde{y}_i(s-t)-1\Vert_{L^2(\widetilde{\varphi})}^2 + \Vert y_i(t)-y_i(r)\Vert_{L^2(\varphi)}^2\Big)\,dt \\
&\leq
16^2\,e^{2T}\,(s-r)^2 + 16\,e^T\,\sup_{0 \leq t \leq T}\Vert\xi_i(t)\Vert_{L^\infty(\varphi)} 
\int_r^s \Big(16e^T (s-t) + 2C (t-r)\Big)\,dt \\
&= 
16e^T\left(16\,e^T + \sup_{0 \leq t \leq T}\Vert\xi_i(t)\Vert_{L^\infty(\varphi)}(8 e^T + C)
\right)\,(s-r)^2
\end{align*}
as long as $0 \leq r \leq s \leq T$ with $s-r \leq 1/16$. Hence we are done. 
\end{proof} 

The next lemma relies heavily on the non-commutative Burkholder--Gundy inequality established by \cite{PisierXu:CMP97}.  

\begin{lemma}\label{L4.9} 
The stochastic integral \eqref{Eq4.2} is well defined even in $L^4(\varphi)$; that is, the net $b_{i,\triangle}(t)$ converges to $b_i(t)$ in $L^4(\varphi)$ as $|\triangle| \to 0$ for every $t \geq 0$. Moreover, for each $T > 0$, there exists a universal constant $C_{2,T}>0$ such that $\Vert b_i(s) - b_i(r)\Vert_{L^4(\varphi)}^4 \leq C_{2,T}(s-r)^2$ holds for any $0 \leq r < s \leq T$.   
\end{lemma} 
\begin{proof} 
We will first establish a rather general estimate. Specifically, for any
\[
\Xi := \sum_{k=1}^m (z_i(s_k)-z_i(s_{k-1})) a_k
\]
with $0 \leq s_0 < s_1 < \cdots < s_m \leq T$ and $a_k \in C_R^*\langle x,u\rangle_{s_{k-1}}$ we will estimate $\Vert \Xi\Vert_{L^2(\varphi)}$ from the above. By \cite[Theorem 2.1]{PisierXu:CMP97} there exists a universal constant $C_0 > 0$ such that 
\[
\Vert \Xi \Vert_{L^4(\varphi)} \leq C_0\,\max\Big\{ \Big\Vert \sum_{k=1}^m \Xi_k \Xi_k^*\Big\Vert_{L^2(\varphi)}^{1/2}, \Big\Vert\sum_{k=1}^m \Xi_k^*  \Xi_k\Big\Vert_{L^2(\varphi)}^{1/2}\Big\}
\] 
with $\Xi_k := (z_i(s_k)-z_i(s_{k-1})) a_k$. For the time being, we set $\alpha := \max\{ \Vert a_k\Vert_{L^\infty(\varphi)}; k=1,\dots,m\}$. We have 
\begin{align*}
&\Big\Vert \sum_{k=1}^m \Xi_k \Xi_k^*\Big\Vert_{L^2(\varphi)}^2 \\
&= 
\sum_{k_1,k_2=1}^m \varphi(\Xi_{k_1}\Xi_{k_1}^*\Xi_{k_2}\Xi_{k_2}^*) \\
&= 
\sum_{1 \leq k_1 < k_2 \leq m} \varphi(\Xi_{k_1}\Xi_{k_1}^*\Xi_{k_2}\Xi_{k_2}^*) + \sum_{1 \leq k_2 < k_1 \leq m} \varphi(\Xi_{k_1}\Xi_{k_1}^*\Xi_{k_2}\Xi_{k_2}^*) + \sum_{k=1}^m \varphi(\Xi_k\Xi_k^*\Xi_k\Xi_k^*) \\
&= 
\sum_{k_2=2}^m \sum_{k_1=1}^{k_2-1} \varphi((z_i(s_{k_2})-z_i(s_{k_2-1}))a_{k_2}a_{k_2}^*(z_i(s_{k_2})-z_i(s_{k_2-1}))^*\Xi_{k_1}\Xi_{k_1}^*) \\
&\quad+\sum_{k_1=2}^m \sum_{k_2=1}^{k_1-1} \varphi((z_i(s_{k_1})-z_i(s_{k_1-1}))a_{k_1}a_{k_1}^*(z_i(s_{k_1})-z_i(s_{k_1-1}))^*\Xi_{k_2}\Xi_{k_2}^*) \\
&\quad+ 
\sum_{k=1}^m \varphi(\Xi_k\Xi_k^*\Xi_k\Xi_k^*) \\
&= 
\sum_{1 \leq k_1 \neq k_2 \leq m}  \Vert a_{k_1}\Vert_{L^2(\varphi)}^2\,\Vert a_{k_2}\Vert_{L^2(\varphi)}^2 \int_{s_{k_1-1}}^{s_{k_1}} e^t\,dt\,\int_{s_{k_2-1}}^{s_{k_2}} e^t\,dt  
+ \sum_{k=1}^m \varphi(\Xi_k\Xi_k^*\Xi_k\Xi_k^*) \\
&\leq 
\sum_{1 \leq k_1 \neq k_2 \leq m}  \Vert a_{k_1}\Vert_{L^2(\varphi)}^2\,\Vert a_{k_2}\Vert_{L^2(\varphi)}^2 \int_{s_{k_1-1}}^{s_{k_1}} e^t\,dt\,\int_{s_{k_2-1}}^{s_{k_2}} e^t\,dt 
+ \alpha^4 \sum_{k=1}^m \Vert z_i(s_k) - z_i(s_{k-1})\Vert_{L^4(\varphi)}^4. 
\end{align*}
By partitioning each interval into smaller subintervals and taking the limit as the mesh size (the maximum length of the subintervals) tends to zero, it follows that 
\[
\Big\Vert \sum_{k=1}^m \Xi_k \Xi_k^*\Big\Vert_{L^2(\varphi)}^{1/2} 
\leq \left(\int_0^T \Big\Vert\sum_{k=1}^m \mathbf{1}_{(s_{k-1},s_k]}(t)a_k\Big\Vert_{L^2(\varphi)}^2\,e^t\,dt\right)^{1/2}.  
\]
holds by Lemma \ref{L4.8}. Similarly, we have
\begin{align*}
&\Big\Vert \sum_{k=1}^m \Xi_k^* \Xi_k\Big\Vert_{L^2(\varphi)}^2 \\
&= 
\sum_{k_2=2}^m \sum_{k_1=1}^{k_2-1} \varphi(a_{k_2}^*(z_i(s_{k_2})-z_i(s_{k_2-1}))^*(z_i(s_{k_2})-z_i(s_{k_2-1}))a_{k_2}\Xi_{k_1}^*\Xi_{k_1}) \\
&\quad+\sum_{k_1=2}^m \sum_{k_2=1}^{k_1-1} \varphi(a_{k_1}^*(z_i(s_{k_1})-z_i(s_{k_1-1}))^*(z_i(s_{k_1})-z_i(s_{k_1-1}))a_{k_1}\Xi_{k_2}^*\Xi_{k_2}) \\
&\quad+ 
\sum_{k=1}^m \varphi(\Xi_k^*\Xi_k\Xi_k^*\Xi_k) \\
&= 
\sum_{k_2=2}^m \sum_{k_1=1}^{k_2-1} \varphi(a_{k_2}\Xi_{k_1}^*\Xi_{k_1}a_{k_2}^*) \int_{s_{k_2-1}}^{s_{k_2}} e^t\,dt 
+\sum_{k_1=1}^m \sum_{k_2=1}^{k_1-1} \varphi(a_{k_1}\Xi_{k_2}^*\Xi_{k_2}a_{k_1}^*) \int_{s_{k_1-1}}^{s_{k_1}} e^t\,dt \\
&\quad+ 
\sum_{k=2}^m \varphi(\Xi_k^*\Xi_k\Xi_k^*\Xi_k) \\
&\leq 
\alpha^2\sum_{k_2=2}^m \sum_{k_1=1}^{k_2-1} \varphi(\Xi_{k_1}\Xi_{k_1}^*) \int_{s_{k_2-1}}^{s_{k_2}} e^t\,dt 
+\alpha^2\sum_{k_1=2}^m \sum_{k_2=1}^{k_1-1}  \varphi(\Xi_{k_2}\Xi_{k_2}^*) \int_{s_{k_1-1}}^{s_{k_1}} e^t\,dt \\
&\quad+ 
\alpha^4 \sum_{k=1}^m \Vert z_i(s_k) - z_i(s_{k-1})\Vert_{L^4(\varphi)}^4 \\
&\leq 
\alpha^2 \sum_{k_2=2}^m \sum_{k_1=1}^{k_2-1} \Vert a_{k_1}\Vert_{L^2(\varphi)}^2 \int_{s_{k_1-1}}^{s_{k_1}} e^t\,dt\, \int_{s_{k_2-1}}^{s_{k_2}} e^t\,dt 
+
\alpha^2\sum_{k_1=2}^m \sum_{k_2=1}^{k_1-1} \Vert a_{k_2}\Vert_{L^2(\varphi)}^2 \int_{s_{k_1-1}}^{s_{k_1}} e^t\,dt\, \int_{s_{k_2-1}}^{s_{k_2}} e^t\,dt \\
&\quad+ \alpha^4 \sum_{k=1}^m \Vert z_i(s_k) - z_i(s_{k-1})\Vert_{L^4(\varphi)}^4.  
\end{align*}
By the same reasoning as the previous estimate, we obtain that 
\[
\Big\Vert \sum_{k=1}^m \Xi_k^* \Xi_k\Big\Vert_{L^2(\varphi)}^{1/2} 
\leq
\left(2\alpha^2(e^{s_m} - e^{s_0}) \int_0^T \Big\Vert\sum_{k=1}^m \mathbf{1}_{(t_{k-1},t_k]}(t)a_k\Big\Vert_{L^2(\varphi)}^2\,e^t\,dt\right)^{1/4}. 
\]
The estimates established so far show that $b_{i,\triangle}(t)$ converges in $L^4(\varphi)$ as $|\triangle|\to0$. Moreover, this limit is shown to coincides with $b_i(t)$. Furthermore, 
\begin{align*}
\Vert b_i(s)-b_i(r)\Vert_{L^4(\varphi)} 
&\leq 
C_0 \max\big\{ \Vert b_i(s)-b_i(r)\Vert_{L^2(\varphi)}, \big(2e^s(e^s - e^r)\,\Vert b_i(s)-b_i(r)\Vert_{L^2(\varphi)}^2\big)^{1/4}\big\} \\
&\leq 
C_0 \max\left\{1,\left(2 e^{2T} \frac{e^T-1}{T}\right)^{1/4}\right\} (s-r)^{1/2}. 
\end{align*}
Hence we are done. 
\end{proof}   

The next lemma is essentially due to Dabrowski (see the proof of \cite[Theorem 22]{Dabrowski:preprint10}). However, since his result is formulated in a more general setting than required for our purpose, we will reconstruct the proof within a restricted framework that specifically fits our problem. 

\begin{lemma} \label{L4.10} 
For every $t \geq 0$, the probability distribution measure of $b_i(t)$ is exactly the centered semicircular distribution of radius $2\sqrt{t}$, and hence $b_i(t)$ belongs to $L^\infty(\varphi)$. 
\end{lemma}
\begin{proof} 
By Lemma \ref{L4.6} the left-multiplication of each $b_i(t)$ on $L^2(\varphi)$ defines a self-adjoint operator affiliated with the tracial $W^*$-algebra obtained by the left-multiplications of $L^\infty(\varphi)$; see Appendix \ref{noncomm_L^2}. Thus, for each $z \in \mathbb{C}$ with $\mathrm{Im}\,z > 0$, $(z1-b_i(t))^{-1} \in L^\infty(\varphi)$ is well defined and $\Vert (z1-b_i(t))^{-1}\Vert_{L^\infty(\varphi)} \leq (\mathrm{Im}\,z)^{-1}$. 

Since
\begin{align*}
\Vert (z1-b_i(s))^{-1} - (z1-b_i(r))^{-1}\Vert_{L^2(\varphi)} 
&= 
\Vert (z1-b_i(s))^{-1}(b_i(s)-b_i(r))(z1-b_i(r))^{-1}\Vert_{L^2(\varphi)} \\
&\leq (\mathrm{Im}\,z)^{-2} \Vert b_i(s)-b_i(r)\Vert_{L^2(\varphi)}
\end{align*}
by the resolvent formula, the map $t \mapsto (z1-b_i(t))^{-1}$ is continuous. 

Let $\triangle: s_0 = 0 < s_1 < \cdots < s_m = t$ be arbitrarily fixed. Observe that 
\begin{align*}
&(z1-b_i(t))^{-1} \\
&= 
(z1-b_i(s_0))^{-1} - (z1-b_i(s_0))^{-1} + (z1-b_i(s_1))^{-1} - (z1-b_i(s_1))^{-1} + \cdots + (z1-b_i(s_m))^{-1} \\
&= 
(z1-b_i(s_0))^{-1} + \sum_{k=0}^{m-1} ((z1-b_i(s_{k+1}))^{-1} - (z1-b_i(s_k))^{-1}) \\
&=  
(z1-b_i(s_0))^{-1} + \sum_{k=0}^{m-1} (z1-b_i(s_{k+1}))^{-1}(b_i(s_{k+1})-b_i(s_k))(z1-b_i(s_k))^{-1}  
\end{align*}
by the resolvent formula. We proceed this pattern of computation using the resolvent formula: 
\begin{align*}
&\sum_{k=0}^{m-1} (z1-b_i(s_{k+1}))^{-1}(b_i(s_{k+1})-b_i(s_k))(z1-b_i(s_k))^{-1} \\ 
&= 
\sum_{k=0}^{m-1} (z1-b_i(s_k))^{-1}(b_i(s_{k+1})-b_i(s_k))(z1-b_i(s_k))^{-1} \\
&\qquad+ 
\sum_{k=0}^{m-1} ((z1-b_i(s_{k+1}))^{-1}-(z1-b_i(s_k))^{-1}(b_i(s_{k+1})-b_i(s_k))(z1-b_i(s_k))^{-1} \\
&= 
\sum_{k=0}^{m-1} (z1-b_i(s_k))^{-1}(b_i(s_{k+1})-b_i(s_k))(z1-b_i(s_k))^{-1} \\
&\qquad+ 
\sum_{k=0}^{m-1} (z1-b_i(s_{k+1}))^{-1}(b_i(s_{k+1})-b_i(s_k))(z1-b_i(s_k))^{-1}(b_i(s_{k+1})-b_i(s_k))(z1-b_i(s_k))^{-1} \\
&= 
\sum_{k=0}^{m-1} (z1-b_i(s_k))^{-1}(b_i(s_{k+1})-b_i(s_k))(z1-b_i(s_k))^{-1} \\
&\qquad+ 
\sum_{k=0}^{m-1} (z1-b_i(s_k))^{-1}(b_i(s_{k+1})-b_i(s_k))(z1-b_i(s_k))^{-1}(b_i(s_{k+1})-b_i(s_k))(z1-b_i(s_k))^{-1} \\
&\qquad+ R_\triangle 
\end{align*}
with 
\begin{align*}
R_\triangle 
&:= 
\sum_{k=0}^{m-1} ((z1-b_i(s_{k+1}))^{-1}-(z1-b_i(s_k))^{-1}) \\
&\phantom{aaaaaaaaaaaaaaa} 
(b_i(s_{k+1})-b_i(s_k))(z1-b_i(s_k))^{-1}(b_i(s_{k+1})-b_i(s_k))(z1-b_i(s_k))^{-1}. 
\end{align*} 

By the martingale property for $b_i(t)$ due to its construction based on Lemma \ref{L4.1} and Lemma \ref{L4.7}, we have 
\begin{align*}
\varphi((z1-b_i(t))^{-1}) 
=
z^{-1} &+ \sum_{k=0}^{m-1} \big(\varphi((z1-b_i(s_k))^{-1})\varphi((z1-b_i(s_k))^{-2})\,(s_{k+1}-s_k) + o(s_{k+1}-s_k)\big) \\
&+ (R_\triangle|1)_{L^2(\varphi)}. 
\end{align*}
Note that we have the following inner product estimate: 
\begin{align*}
&\big|((b_i(s_{k+1})-b_i(s_k))(z1-b_i(s_k))^{-1}(b_i(s_{k+1})-b_i(s_k))(z1-b_i(s_k))^{-1} \\
&\phantom{aaaaaaaaaaaaaaaaaaaaaaaaaaaaaaaaaaaaaaa}
|\,((z1-b_i(s_{k+1}))^{-1}-(z1-b_i(s_k))^{-1})^*)_{L^2(\varphi)}\big| \\
&= 
\big|((b_i(s_{k+1})-b_i(s_k))(z1-b_i(s_k))^{-1}(b_i(s_{k+1})-b_i(s_k))(z1-b_i(s_k))^{-1} \\
&\phantom{aaaaaaaaaaaaaaaaaaaaaa}
|\,((z1-b_i(s_{k+1}))^{-1}(b_i(s_{k+1} - b_i(s_k))(z1-b_i(s_k))^{-1})^*)_{L^2(\varphi)}\big| \\
&\leq 
\Vert ((b_i(s_{k+1})-b_i(s_k))(z1-b_i(s_k))^{-1}(b_i(s_{k+1})-b_i(s_k))(z1-b_i(s_k))^{-1})\Vert_{L^2(\varphi)} \\
&\phantom{aaaaaaaaaaaaaaaaaaaaaa} \times 
\Vert (z1-b_i(s_{k+1}))^{-1}(b_i(s_{k+1} - b_i(s_k))(z1-b_i(s_k))^{-1}\Vert_{L^2(\varphi)} \\
&\leq 
(\mathrm{Im}\,z)^{-4}\,\Vert b_i(s_{k+1})-  b_i(s_k)\Vert_{L^4(\varphi)}^2 \Vert b_i(s_{k+1})-  b_i(s_k)\Vert_{L^2(\varphi)},  
\end{align*}
Combined with the identity \ref{Eq4.3} and Lemma \ref{L4.9}, this leads to the following estimate: 
\begin{align*} 
|\varphi(R_\triangle)| 
&\leq 
(\mathrm{Im}\,z)^{-4}\sum_{k=0}^{m-1} \Vert b_i(s_{k+1})-  b_i(s_k)\Vert_{L^2(\varphi)} \Vert b_i(s_{k+1})-  b_i(s_k)\Vert_{L^4(\varphi)}^2 \\
&\leq 
(\mathrm{Im}\,z)^{-4}\sum_{k=0}^{m-1} (s_{k-1}-s_k)^{1/2} C_{t,2}^{1/2} (s_{k+1}-s_k) \\
&\leq 
(\mathrm{Im}\,z)^{-4} |\triangle|^{1/2}\,C_{t,2}^{1/2} t \to 0 
\end{align*}
as $|\triangle| \to 0$. Consequently, we obtain that 
\[
\varphi((z1-b_i(t))^{-1}) = z^{-1} + \int_0^t \varphi((z1-b_i(s))^{-1})\varphi((z1-b_i(s))^{-2})\,ds. 
\]
This means that $G(t,z) := \varphi((z1-b_i(t))^{-1})$ satisfies the quasi-linear partial differential equation (the same form as inviscid Burger's equation)
\[
\partial_t G(t,z) + G(t,z) \partial_z G(t,z) = 0, \qquad G(0,z) = z^{-1}. 
\]
It is well known that one can solve this by the method of characteristics (see e.g.\ \cite[subsection 3.2.2.b]{Evans:Book}). The characteristic curve is given by $z(t) = z + t/z$, that is, $G(t,z(t)) = G(0,z) = 1/z$ holds. The standard knowledge on the Joukowski transform shows that $z \mapsto z+ t/z$ gives a conformal equivalence from $\mathbb{H}^+\setminus\overline{B(0,\sqrt{t})}$ onto $\mathbb{H}^+$, where $\mathbb{H}^+$ denotes the upper half plane $\mathrm{Im}\,z > 0$ and $B(0,\sqrt{t})$ the open ball centered at $0$ of radius $\sqrt{t}$. Using its inverse map $z \mapsto (z+\sqrt{z^2-4t})/2$, we obtain         
\[
G(t,z) = \left(\frac{z+\sqrt{z^2 - 4t}}{2}\right)^{-1} = 
\frac{z-\sqrt{z^2-4t}}{2t}, \qquad (t,z) \in [0,+\infty)\times\mathbb{H}^+.  
\]
from which the desired assertion immediately follows thanks to the Stieltjes inversion formula (see e.g.\ \cite[section 3.1]{MingoSpeicher:Book}).
\end{proof}

We are now in position to apply Biane--Capitaine--Guionnet's L\'{e}vy-type characterization of free Brownian motion to our stochastic integral $b(t) = (b_1(t),\dots,b_n(t))$. Lemmas \ref{L4.6} and \ref{L4.10} guarantee that the $b(t)$ is an $n$-dimensional self-adjoint $L^\infty(\varphi)$-martingale, that is, item 1 of \cite[Theorem 6.2]{BianeCapitaineGuionnet:InventMath03} holds. Lemma \ref{L4.7} is nothing less than item 3 of \cite[Theorem 6.2]{BianeCapitaineGuionnet:InventMath03}. Moreover, Lemma \ref{L4.9} guarantees item 2 of \cite[Theorem 6.2]{BianeCapitaineGuionnet:InventMath03}. Consequently, we arrive at the next assertion. 

\begin{proposition}\label{P4.11}
The stochastic integral \eqref{Eq4.2} $b(t) = (b_1(t),\dots,b_n(t))$ defines an $n$-dimensional free Brownian motion adapted to the filtration $(L^\infty(\varphi)_t)_{t \geq 0}$.  
\end{proposition}

Here is the main theorem of this section.  

\begin{theorem}\label{T4.12} 
The non-commutative stochastic process $y(t) = (y_1(t),\dots,y_n(t))$ satisfies the free SDE 
\[
dy_i(t) = \sqrt{-1}(db_i(t)+\xi_i(t)\,dt)y_i(t), \qquad i=1,\dots,n
\]
with the $n$-dimensional free Brownian motion $b(t) = (b_1(t),\dots,b_n(t))$ in Proposition \ref{P4.11}. 
\end{theorem}
\begin{proof}
For the time being, we fix a division $\triangle : 0 = s_0 < s_1 < \cdots < s_m = t$ for a given $t > 0$. By the martingale property for  $z_i(t)$ (Lemma \ref{L4.1}) as well as for $b_i(t)$ due to its construction, we have 
\begin{align*} 
&\left\Vert z_i(t) - \sqrt{-1}\sum_{k=1}^m (b_i(s_k) - b_i(s_{k-1}))y_i(s_{k-1})\right\Vert_{L^2(\varphi)}^2 \\
&\quad= 
\left\Vert \sum_{k=1}^m \big((z_i(s_k)-z_i(s_{k-1}) - \sqrt{-1}(b_i(s_k)-b_i(s_{k-1}))y_i(s_{k-1})\big)\right\Vert_{L^2(\varphi)}^2 \\
&\quad= 
\sum_{k=1}^m \Vert (z_i(s_k)-z_i(s_{k-1}) - \sqrt{-1}(b_i(s_k)-b_i(s_{k-1}))y_i(s_{k-1})\Vert_{L^2(\varphi)}^2 \\
&\quad= 
\sum_{k=1}^m \int_{s_{k-1}}^{s_k} \Vert 1 - y_i(s)^{-1}y_i(s_{k-1})\Vert_{L^2(\varphi)}^2\,e^s\,ds \\
&\quad\leq 
\sum_{k=1}^m \int_{s_{k-1}}^{s_k} \Vert y_i(s) - y_i(s_{k-1})\Vert_{L^2(\varphi)}^2\,e^s\,ds \\
&\quad\leq 
\max_{k=1,\dots,m} \sup_{s_{k-1}\leq s \leq s_k} \Vert y_i(s) - y_i(s_{k-1})\Vert_{L^2(\varphi)}^2\,(e^t-1).  
\end{align*}
By taking the limit of this inequality as $|\triangle| \to 0$, we have 
\[
\left\Vert z_i(t) - \sqrt{-1}\int_0^t db_i(s)\,y_i(s)\right\Vert_{L^2(\varphi)}^2 = 0.
\]
Hence we obtain that 
\[
y_i(t) - \sqrt{-1}\int_0^t \xi_i(s)y_i(s)\,ds = \sqrt{-1}\int_0^t db_i(s)\,y_i(s), 
\]
which shows the desired fact. 
\end{proof}

\begin{corollary}\label{C4.13} 
The unitary stochastic process $u(t) = (u_1(t),\dots,u_n(t))$ satisfies the free SDE \eqref{freeSDE}
\[
du_i(t) = \sqrt{-1}(db_i(t)+\xi_i(t)\,dt)u_i(t) - \frac{1}{2}u_i(t)\,dt, \qquad i=1,\dots,n
\]
with the $n$-dimensional free Brownian motion $b(t) = (b_1(t),\dots,b_n(t))$ in Proposition \ref{P4.11}. 
\end{corollary}
\begin{proof}
Since $u_i(t) = e^{-t/2}y_i(t)$, this corollary follows from the previous theorem in conjunction with free It\^{o} formula (see \cite{BianeSpeicher:PTRF98}). 
\end{proof} 

\section{The Brownian motion on the unitary group}\label{applications_to_unitary_BM}  

Following the method in \cite{BianeCapitaineGuionnet:InventMath03} we will prove a weak large deviation lower bound for the $n$-dimensional (left-increment) Brownian motion $U_N(t) = (U_{N,1}(t),\dots,U_{N,n}(t))$ on the unitary group $\mathrm{U}(N)$. This process is constructed from an $n$-dimensional $N\times N$ self-adjoint matrix Brownian motion $H_N(t) = (H_{N,1}(t),\dots,H_{N,n}(t))$ given by:  
\[
H_{N,i}(t) = \sum_{\alpha,\beta=1}^N \frac{B_{N,i;\alpha,\beta}(t)}{\sqrt{N}}C_{\alpha,\beta}, \quad t \geq 0,\ i=1,\dots,n, 
\]
where $(B_{N,i;\alpha,\beta}(t))_{i=1,\dots,n;\alpha,\beta=1,\dots,N}$ is a standard $nN^2$-dimensional Brownian motion with the natural filtration $(\mathcal{F}_t)_{t \geq 0}$. The matrices $\{C_{\alpha,\beta}; \alpha,\beta=1,\dots,N\}$ form an orthonormal basis of the Euclidean space $M_N(\mathbb{C})^\mathrm{sa}$ equipped with the Hilbert-Schmidt inner product $\langle A,B\rangle_{HS} = \mathrm{Tr}_N(AB)$. 

As pointed out in \cite[Subsection 7.2]{Ueda:CJM21}, for any $a=a^* \in \mathbb{C}\langle x,u\rangle$ and each $t > 0$, we have:  
\begin{equation}\label{Eq5.1} 
\begin{aligned}
\mathbb{E}[\varphi_N(a)|\mathcal{F}_t] 
&= 
\mathbb{E}[\varphi_N(a)] + \int_0^t \sum_{i=1}^n\sum_{\alpha,\beta=1}^N \mathbb{E}\left[\mathrm{tr}_N\left(\pi_N(\mathfrak{D}_{s,i}a)\left(\frac{1}{\sqrt{N}}C_{\alpha,\beta}\right)\right)\Big|\,\mathcal{F}_s\right]\,dB_{N,i;\alpha,\beta}(s), 
\end{aligned}
\end{equation} 
where $\pi_N = \pi_{(X_N,U_N)} : C^*_R\langle x,u\rangle \to M_N(\mathbb{C})$ is the random unital $*$-homomorphism, $\varphi_N := \varphi_{(X_N,U_N)} = \mathrm{tr}_N\circ\pi_N \in TS^c(C^*_R\langle x,u\rangle)$ is the empirical distribution, and $\mathrm{tr}_N = N^{-1}\mathrm{Tr}_N$ (see Subsection 2.4 for these notations). 

\subsection{Smooth continuous tracial states} 

Let $c = c^* \in \mathbb{C}\langle x,u\rangle$ be arbitrarily given. For each $N \in \mathbb{N}$ we consider the martingale $M_{c,N}(t)$ defined by: 
\[
M_{c,N}(t) := N^2\left(\mathbb{E}[\varphi_N(c)|\mathcal{F}_t] - \mathbb{E}[\varphi_N(c)]\right) = N^2\left( \mathrm{tr}_N(\mathbb{E}[\pi_N(c)|\mathcal{F}_t])  - \mathrm{tr}_N(\mathbb{E}[\pi_N(c)])\right),  
\]
where $\mathbb{E}\left[\pi_N(\mathfrak{D}_{s,i}c)|\,\mathcal{F}_s\right]$ is interpreted in a matrix-entrywise sense. By the Clark--Ocone formula \eqref{Eq5.1} we observe that 
\[
M_{c,N}(t) 
= 
N^2 \int_0^t \sum_{i=1}^n\sum_{\alpha,\beta=1}^N \mathrm{tr}_N\left(\mathbb{E}[\pi_N(\mathfrak{D}_{s,i}c)|\,\mathcal{F}_s]\left(\frac{1}{\sqrt{N}}C_{\alpha,\beta}\right)\right)\,dB_{N,i;\alpha,\beta}(s). 
\] 
Hence we can compute the quadratic variation as
\begin{align*}
\langle M_{c,N}\rangle(t) 
&= 
N^4 \int_0^t \sum_{i=1}^n \sum_{\alpha,\beta=1}^N \mathrm{tr}_N\left(\mathbb{E}[\pi_N(\mathfrak{D}_{s,i}c)|\,\mathcal{F}_s]\left(\frac{1}{\sqrt{N}}C_{\alpha,\beta}\right)\right)^2\,ds \\
&=
N \int_0^t \sum_{i=1}^n \sum_{\alpha,\beta=1}^N \left\langle  \mathbb{E}\left[\pi_N(\mathfrak{D}_{s,i}c)|\,\mathcal{F}_s\right],C_{\alpha,\beta}\right\rangle_{HS}^2\,ds \\
&= 
N  \int_0^t \sum_{i=1}^n \left\Vert\mathbb{E}\left[\pi_N(\mathfrak{D}_{s,i}c)|\,\mathcal{F}_s\right]\right\Vert_{HS}^2\,ds \\
&= 
N^2  \int_0^t \sum_{i=1}^n \mathrm{tr}_N(\mathbb{E}\left[\pi_N(\mathfrak{D}_{s,i}c)|\,\mathcal{F}_s\right]^2)\,ds 
\end{align*}
with $\Vert A\Vert_{\mathrm{tr}_N,2} := \mathrm{tr}_N(A^* A)^{1/2}$. Set 
\begin{equation*}
I_{c,N}(t) := \mathrm{tr}_N(\mathbb{E}[\pi_N(c)|\,\mathcal{F}_t]) - \mathrm{tr}_N(\mathbb{E}[\pi_N(c)]) - \frac{1}{2}\int_0^t \sum_{i=1}^n \mathrm{tr}_N(\mathbb{E}\left[\pi_N(\mathfrak{D}_{s,i}c)|\,\mathcal{F}_s\right]^2)\,ds.  
\end{equation*}
It is not difficult to see that for each $T>0$ there is a constant $C_T>0$ such that $|I_{c,N}(t)| \leq C_T$ for all $0 \leq t \leq T$. Thus, we see, by e.g.\ \cite[Corollary 3.5.13]{KaratzasShreve:Book}, that $\exp(N^2 I_{c,N}(t))$ is a bounded martingale (over any finite interval), and moreover,
\begin{equation}\label{Eq5.2} 
0 < \exp(N^2 I_{c,N}(t)) \leq e^{N^2 C_T}, \qquad 0 \leq t \leq T
\end{equation}
for each $T > 0$. It follows that for each $t > 0$, $d\mathbb{P}_{c,t,N} := \exp(N^2 I_{c,N}(t))\,d\mathbb{P}$ defines a probability measure that is equivalent to the original $\mathbb{P}$. 

In what follows, we choose and fix $T > 0$ such that $c \in \mathbb{C}\langle x, u\rangle_T$.  Consider any sequence of couplings of $\mathbb{P}_{c,T,N}$ for $N \in \mathbb{N}$, which we fix (such a sequence indeed exists!) and denote by $\mathbb{P}_c$ hereafter.  By the estimate \eqref{Eq5.2} and the disccusion in \cite[subsection 5.1]{Ueda:JOTP19} we see that the probability distributions of $\varphi_N$ under $\mathbb{P}_c$ --namely, $\Lambda \mapsto \mathbb{P}_c(\varphi_N \in \Lambda) = \mathbb{P}_{c,T,N}(\varphi_N \in \Lambda) = \mathbb{E}[\exp(N^2 I_{c,N}(T))\mathbf{1}_{(\varphi_N \in \Lambda)}]$ -- form an exponentially tight sequence. By the standard Borel--Cantelli argument, it follows that there exists a compact subset $\mathcal{K} \subset TS^c(C^*_R\langle x,u\rangle)$ such that $\mathbb{P}_c(\text{$\varphi_N \in \mathcal{K}$ for sufficiently large $N$}) = 1$, that is, $\{\varphi_N\}$ is almost surely relatively compact in $TS^c(C^*_R\langle x,u\rangle)$ under $\mathbb{P}_c$.  

Similarly to the calculation at the beginning of this subsection, we have 
\begin{equation}\label{Eq5.3}
\begin{aligned}
\mathrm{tr}_N(\mathbb{E}[\pi_N(a)|\,\mathcal{F}_t]) 
&= 
\mathrm{tr}_N(\mathbb{E}[\pi_N(a)]) 
+ 
\int_0^t \sum_{i=1}^n \sum_{\alpha,\beta=1}^N \mathrm{tr}_N(\mathbb{E}[\pi_N(\mathfrak{D}_{s,i}a)|\,\mathcal{F}_s] \mathbb{E}[\pi_N(\mathfrak{D}_{s,i}c)|\,\mathcal{F}_s])\,ds \\
&\quad+ \int_0^t \sum_{i=1}^n \sum_{\alpha,\beta=1}^N \mathrm{tr}_N\left(\mathbb{E}\left[\pi_N(\mathfrak{D}_{s,i}c)|\,\mathcal{F}_s\right]\left(\frac{1}{\sqrt{N}}C_{\alpha,\beta}\right)\right)^2\,dB_{c,N,i;\alpha,\beta}(s),   
\end{aligned}
\end{equation} 
where
\begin{equation*} \label{Eq5.x}
B_{c,N,i;\alpha,\beta}(t) := B_{N,i;\alpha,\beta}(t) - \langle B_{N,i;\alpha,\beta}, M_{c,N}\rangle(t), \qquad i=1,\dots,n,\quad \alpha,\beta = 1,\dots,N
\end{equation*}
form an $nN^2$-dimensional Brownian motion over the time interval $[0,T]$ under $\mathbb{P}_c$ thanks to the celebrated Cameron--Martin--Maruyama--Girsanov theorem (see e.g., \cite[Chapter 3, Section 3.5]{KaratzasShreve:Book}).

\begin{remark}\label{R5.1} 
{\rm We obtain the process $B_{c,N,i;\alpha,\beta}(t)$ defined by  
\[
B_{c,N,i;\alpha,\beta}(t) = B_{N,i;\alpha,\beta}(t) - \int_0^t \mathrm{tr}_N\left(\mathbb{E}[\pi_N(\mathfrak{D}_{s,i}c)|\,\mathcal{F}_s]\left(\frac{1}{\sqrt{N}}C_{\alpha,\beta}\right)\right)\,ds. 
\]
Consequently, the $n$-dimensional $N\times N$ self-adjoint matrix Brownian motion $H_{c,N}(t) = (H_{c,N,1}(t),\dots,H_{c,N,n}(t))$, where each component is given by 
\[
H_{c,N,i}(t) := \sum_{\alpha,\beta=1}^N \frac{B_{c,N,i;\alpha,\beta}(t)}{\sqrt{N}} C_{\alpha,\beta}, \qquad t \geq 0, \quad i=1,\dots,n
\]
under $\mathbb{P}_c$, yields the following SDE:  
\[
dU_{N,i}(t) = \sqrt{-1}\,dH_{c,N,i}(t)\,U_{N,i}(t) + \sqrt{-1}\,\mathbb{E}[\pi_N(\mathfrak{D}_{t,i}c)|\,\mathcal{F}_t]\,U_{N,i}(t)\,dt - \frac{1}{2}U_{N,i}(t)\,dt. 
\]
Note that the term $\mathbb{E}[\pi_N(\mathfrak{D}_{t,i}c)|\,\mathcal{F}_t]$ can be interpreted as the entry-wise expectation (with respect to $V_N(t)$) of the process 
\[
s \mapsto U_{N,i}^t(s) := V_{N,i}((s-t)\vee0)U_{N,i}(t\wedge s), \qquad i=1,\dots,n, 
\]
where $V_N(t) = (V_{N,1}(t),\dots,V_{N,n}(t))$ is an $n$-dimensional unitary Brownian motion independent of the current filtration.  
}
\end{remark}

We denote by $M^0_N(t)$ the martingale part of the right-hand side of the identity \eqref{Eq5.3}. 

\begin{lemma}\label{L5.2} $\displaystyle \lim_{N\to\infty} \sup_{t \geq 0} |M^0_N(t)| = 0$ almost surely under $\mathbb{P}_c$. \end{lemma} 
\begin{proof} 
Write $M_N^\star(T) := \sup_{0 \leq t \leq T} |M_N^0(t)|$ for simplicity. By Chebyshev's inequality and the Burkholder--Davis--Gundy theorem (see e.g., \cite[Theorem 3.3.28]{KaratzasShreve:Book}) there exists a constant $C > 0$ so that 
\[
\varepsilon^2 \mathbb{P}_c(M_N^\star(T) \geq \varepsilon) \leq C\,\mathbb{E}_c[\langle M_N^0\rangle(T)]
\]
for all $\varepsilon>0$, where $\mathbb{E}_c$ denotes the expectation with respect to $\mathbb{P}_c$. Analogous to the estimate \eqref{Eq5.2} we observe that the quadratic variation satisfies:
\[
\langle M_N^0\rangle(T) = \frac{1}{N^3} \int_0^T \sum_{i=1}^n \Vert \mathbb{E}[\pi_N(\mathfrak{D}_{s,i}c)|\,\mathcal{F}_s]\Vert_{HS}^2\,ds \leq \frac{C'}{N^2}
\]
for some $C' > 0$. Therefore, we obtain
\[
\mathbb{P}_c(M^\star_N(T) \geq \varepsilon) \leq \frac{CC'}{\varepsilon^2 N^2}. 
\]
From this estimate, a standard Borel--Cantelli argument implies that $\lim_{N\to\infty} M_N^\star(T) = 0$ as $N \to \infty$ almost surely under $\mathbb{P}_c$. Since $M_N^0(t) = M_N^0(T)$ holds for all $t > T$, we conclude the desired assertion.  
\end{proof}

Based on the preceding discussion, both the relative compactness of $\{\varphi_N\}$ in $TS^c(C^*_R\langle x,u\rangle)$ and the almost sure convergence of $\sup_{t\geq 0}|M_N^0(t)|$ to $0$ as $N\to\infty$ hold simultaneously under $\mathbb{P}_c$. 

For any sample path within this event of probability one, let $\varphi \in TS^c(C^*_R\langle x,u\rangle)$ be a limit point of $\{\varphi_N\}$; specifically, there exists a subsequence $\varphi_{N_k}$ that converges to $\varphi$ in $TS^c(C^*_R\langle x,u\rangle)$ as $k\to\infty$. 

What we actually proved in \cite[section 4]{Ueda:JOTP19}, when applied to the current setup, shows that 
\begin{equation}\label{Eq5.4}
\lim_{k\to\infty}\sup_{t \geq 0} |\mathrm{tr}_{N_k}(\mathbb{E}[\pi_{N_k}(\mathfrak{D}_{t,i}a)|\,\mathcal{F}_t]^m) - \varphi(E^{\varphi^t}_t(\mathfrak{D}_{t,i}a)^m)| = 0 
\end{equation}
for every $a = a^* \in \mathbb{C}\langle x,u\rangle$, $i=1,\dots,n$ and $m \in \mathbb{N}$. Consequently, in conjunction with Lemma \ref{L3.4}, the formula \eqref{Eq5.3} implies that $\varphi$ satisfies item (ii) in Theorem \ref{T3.12}. By Corollary \ref{C4.13} we conclude that the free SDE \eqref{freeSDE} holds in $L^2(\varphi)$ with the drift term $\xi(t) = (E_t^\varphi(\mathfrak{D}_{t,1}c),\dots,E_t^\varphi(\mathfrak{D}_{t,n}c))$. 

Consider $w = u_i(t_1)x_1 u_i(t_2) x_2 \in \mathbb{C}\langle x,u\rangle$ with $0 < t_1 < t_2$ for example.
A tedious calculation shows 
\begin{align*}
E_t^{\varphi^t}(\mathfrak{D}_{t,i}w) 
&= 
\sqrt{-1}\,\mathbf{1}_{[0,t_1]}(t)\,e^{-2^{-1}((t_1-s)\vee0 + (t_2 - t)\vee0)} u_i(t_1 \wedge t) x_1 u_i(t_2\wedge t)x_2 \\
&\quad+
\sqrt{-1}\,\mathbf{1}_{[0,t_2]}(t)\,e^{-2^{-1}((t_1-t)\vee0 + (t_2 - t)\vee0)} u_i(t_2\wedge t)x_2 u_i(t_1 \wedge t) x_1 \\ 
&\quad 
- ((t_1 - t)\vee0)\,\varphi(u_i(t_2\wedge s)x_t) u_i(t_1\wedge t)x_1 \\
&\quad 
- ((t_1 - t)\vee0)\,\varphi(u_i(t_1\wedge t)x_1) u_i(t_2\wedge t)x_2. 
\end{align*}
In this way, we see that $E_t^{\varphi^t}(\mathfrak{D}_{t,i} c)$ is a linear combination of monomials of the form 
\[
x_{j_0} u_{i_1}(t_1 \wedge t) x_{j_1} u_{i_2}(t_2 \wedge t) \cdots u_{i_m}(t_m\wedge t) x_{j_m}
\]
with piecewise-continuous (in $t$) coefficients. We denote the drift term as $f_i(t,u) := E_t^{\varphi_t}(\mathfrak{D}_{s,i}c)$, where the notation emphasizes its dependence on $u$. For a fixed potential $c$, this $f_i$ is a ``functional" that depends on the trajectory of the process $u$ up to time $t$) as well as on $t$ itself. Consequently, our free SDE \eqref{freeSDE} is 
\begin{equation}\label{Eq5.5}
du_i(t) = \sqrt{-1}\,db_i(t) u_i(t) + \left(\sqrt{-1}\,f_i(t,u) - \frac{1}{2}u_i(t)\right)\,dt, \qquad i=1,\dots,n.
\end{equation}

Here is a proposition, which is just a special case of the natural free probability counterpart of the well-known theorem in stochastic analysis.  

\begin{proposition}\label{P5.3} 
The free SDE \eqref{freeSDE} with the drift term $\xi(t) = (E_t^\varphi(\mathfrak{D}_{t,1}c),\dots,E_t^\varphi(\mathfrak{D}_{t,n}c))$ possesses a unique solution. Furthermore, this solution is adapted to the filtration $(W^*(x, b[0,t]))_{t\geq0}$ rather than the full filtration $(L^\infty(\varphi)_t)_{t\geq0}$. Here, $W^*(x,b[0,t])$ denotes the $W^*$-subalgebra of $L^\infty(\varphi)_t$ generated by the initial elements $x_j$ and the free Brownian motion $b_i(s)$ for $s \leq t$.    
\end{proposition} 
\begin{proof} 
We want to solve the free SDE \eqref{Eq5.5} by the usual Picard iteration procedure. To this end, we need to ``truncate" the function $f_i(u,t)$ following Biane--Speicher's idea \cite[Theorem 3.1]{BianeSpeicher:AIHP01}. 

Choose a continuous function $g(t)$ such that $g(t) = 0$ if either $0 \leq t \leq 1$ or $t \geq 2$, $0 \geq g(t) \geq -2$ and $\int_1^2 g(t)\,dt = -1$. Set $h(x) := 1+ \int_0^x g(t)\,dt$, $x \geq 0$. Then $h(x)$ is globally Lipschitz with a Lipschitz constant $2$.   
Define 
\[
\tilde{f}_i(t,u) := h\big(\sup_{0 \leq s \leq t} \Vert u_1(s)\Vert_{L^\infty(\varphi)}\big)\cdots h\big(\sup_{0 \leq s \leq t} \Vert u_n(s)\Vert_{L^\infty(\varphi)}\big)\,f_i(u,t). 
\]
If some of $u_i(s)$ with $s \leq t$ has the $L^\infty$-norm greater than $2$, then $\tilde{f}_i(u,t)$ must be $0$. By this fact, it is easy to see that 
\[
\Vert \tilde{f}_i(t,u) - \tilde{f}_i(t,v)\Vert_{L^\infty(\varphi)} \leq C \max_{i=1,\dots,n} \sup_{0 \leq s \leq t} \Vert u_i(s) - v_i(s)\Vert_{L^\infty(\varphi)}
\]
holds for some $C > 0$. Moreover, we can control the coefficient of $db_i(t)$ by the free Burkholder--Gundy inequality due to Biane--Speicher \cite[subsection 3.2]{BianeSpeicher:PTRF98}; thus, the Picard iteration can be applied to the integral equation
\[
u_i(t) = 1 + \sqrt{-1}\int_0^t db_i(s)\,u_i(s) + \int_0^t \sqrt{-1}\,\widetilde{f}_i(s,u)u_i(s) - \frac{1}{2}u_i(s)\,ds. 
\]
In this way, we can establish the existence of solution to the free SDE \eqref{Eq5.5} with $\tilde{f}_i$ in place of $f_i$ adapted to the filtration $(W^*(x,b[0,t]))_{t \geq 0}$. Since $f_i(t,u)^* = f_i(t,u)$ (and hence $\tilde{f}_i(t,u)^* = \tilde{f}_i(t,u)$ too), the free It\^{o} formula (see \cite{BianeSpeicher:PTRF98}) enables us to 
see that $u_i(t)^*u_i(t) = 1$ for all $t \geq 0$. This says that all $u_i(t)$ must be unitary elements and thus $\tilde{f}_i(t,u) = f_i(t,u)$ ($ = E_t^{\varphi^t}(\mathfrak{D}_{t,i}c)$) holds for any solutions $u(t)$. The uniqueness of solution to the free SDE \eqref{Eq5.5} clearly holds in the full filtration $(L^\infty(\varphi)_t)_{t \geq 0}$ too.  
\end{proof}

The description of $E_t^{\varphi^t}(\mathfrak{D}_{t,i}c)$ falls within the notion of ``trace polynomials", a phenomenon that was likely first explicitly identified by C\'{e}bron \cite{Cebron:JFA13}. In this respect, Jekel's detailed investigation into trace polynomials (see e.g., his thesis \cite{Jekel:thesis}) provides a valuable foundation for further study in this direction. 
 
The following is one of the main results of this subsection. 
 
\begin{theorem}\label{T5.4}  
Given a self-adjoint potential $c = c^* \in \mathbb{C}\langle x, u\rangle$, there exists a unique continuous tracial state $\varphi_c \in TS^c(C_R^*\langle x,u\rangle)$ such that item (ii) of Theorem \ref{T3.12} is satisfied with $\varphi_c$ and the drift term $\xi(t) = (E_t^{\varphi_c^t}(\mathfrak{D}_{t,1}c),\dots,E_t^{\varphi_c^t}(\mathfrak{D}_{t,n}c))$. Moreover, the empirical distribution $\varphi_N = \varphi_{(X_N,U_N)}$ converges to $\varphi_c$ in $TS^c(C^*_R\langle x,u\rangle)$ almost surely under $\mathbb{P}_c$. 
\end{theorem}

We denote by $TS^\omega(C^*_R\langle x,u\rangle)$ all the tracial states $\varphi_c$ determined from $c = c^* \in \mathbb{C}\langle x,u\rangle$ thanks to this theorem. 

\begin{proof}
The existence of $\varphi_c$ has already been established. Given the preceding discussion, the almost sure convergence of $\varphi_N$ to $\varphi_c$ under $\mathbb{P}_c$ follows immediately from the uniqueness of $\varphi_c$. Thus, it remains only to prove the uniqueness of $\varphi_c$. 

Assume that there is another $\psi \in TS^c(C^*_R\langle x,u\rangle)$ such that item (ii) of Theorem \ref{T3.12} holds with $\psi$ and $\xi(t) = (E_t^{\psi^t}(\mathfrak{D}_{t,1}c),\dots,E_t^{\psi^t}(\mathfrak{D}_{t,n}c))$. By Corollary \ref{C4.13} there exists an $n$-dimensional free Brownian motion $b'(t) = (b'_1(t),\dots,b'_n(t))$ adapted to the filtration $L^\infty(\psi)_t$, $t \geq 0$, such that the free SDE 
\[
du_i(t) = \sqrt{-1}\,(db'_i(t) + \sqrt{-1}\,E_t^{\psi_t}(\mathfrak{D}_{t,i}c)dt)u_i(t) - \frac
{1}{2}u_i(t)\,dt, \qquad i=1,\dots,n
\] 
holds in $L^\infty(\psi)$. By Proposition \ref{P5.3} we have a trace-preserving isomorphism from $L^\infty(\psi) = W^*(x,b')$ to $L^\infty(\varphi_c) = W^*(x,b)$ that maps $x_j$ and $b'_i(t)$ in $L^\infty(\psi)$ to $x_j$ $b_i(t)$ in $L^\infty(\varphi_c)$, respectively. This isomorphism transforms the free SDE \eqref{Eq5.5} with $b'_i(t)$ and $\psi$ to that with $b_i(t)$ and $\varphi_c$. The uniqueness of solution to the free SDE in Proposition \ref{P5.3} shows that the isomorphism maps the solution in $L^\infty(\psi)$ to the one in $L^\infty(\varphi_c)$. It turns out that $\psi = \varphi_c$ in $TS^c(C^*_R\langle x,u\rangle)$ holds.     
\end{proof}

\subsection{A weak large deviation result} 
Theorem \ref{T3.12} leads to the following identity: 
\begin{equation*}
I(\varphi_c) = \frac{1}{2}\int_0^\infty \sum_{i=1}^n \Vert E_t^{\varphi_c^t}(\mathfrak{D}_{t,i}c)\Vert_{L^2(\varphi_c)}^2\,dt
\end{equation*}
for any $c = c^* \in \mathbb{C}\langle x,u\rangle$. 

\begin{proposition}\label{P5.5} 
For any $c = c^* \in \mathbb{C}\langle x,u \rangle$,  
\[
\lim_{\varepsilon\searrow0}\varliminf_{N\to\infty} \frac{1}{N^2}\log\mathbb{P}(d(\varphi_N,\varphi_c) < \varepsilon) \geq -I(\varphi_c)
\]
holds. 
\end{proposition}
\begin{proof} 
Choose a $T > 0$ such that $c \in \mathbb{C}\langle x,u\rangle_T$. We have 
\begin{align*}
\mathbb{P}(d(\varphi_N,\varphi_c) < \varepsilon) 
&= 
\mathbb{E}\left[\mathbf{1}_{(d(\varphi_B,\varphi_c)<\varepsilon)} \exp(N^2 I_{c,N}(T) - N^2 I_{c,N}(T))\right] \\
&\geq 
\mathbb{P}_{c,T,N}(d(\varphi_N,\varphi_c) < \varepsilon)\,\exp\left(-N^2 \mathrm{esssup}\big\{ I_{c,N}(T)\,\big|\, d(\varphi_N,\varphi_c) < \varepsilon\big\}\right), 
\end{align*}
and hence 
\begin{align*} 
\frac{1}{N^2} \log\mathbb{P}(d(\varphi_N,\varphi_c) < \varepsilon) 
&\geq 
\frac{1}{N^2} \log\mathbb{P}_c(d(\varphi_N,\varphi_c) < \varepsilon) - \mathrm{esssup}\big\{ I_{c,N}(T)\,\big|\, d(\varphi_N,\varphi_c) < \varepsilon\big\}. 
\end{align*}
Since $d(\varphi_N,\varphi_c) \to 0$ as $N\to\infty$ almost surely under $\mathbb{P}_c$ by Theorem \ref{T5.4}, the first term on the right-hand side of the above inequality converges to $0$ as $N\to\infty$. Consequently, we obtain 
\[ 
\varliminf_{N\to\infty} \frac{1}{N^2} \log\mathbb{P}(d(\varphi_N,\varphi_c) < \varepsilon) \geq 
- \varlimsup_{N\to\infty} \mathrm{esssup}\big\{ I_{c,N,T}\,\big|\, d(\varphi_N,\varphi_c) < \varepsilon\big\}. 
\] 
Then, by the convergence result \eqref{Eq5.4} we conclude  
\[
\lim_{\varepsilon\searrow0}\varliminf_{N\to\infty} \frac{1}{N^2}\log\mathbb{P}(d(\varphi_N,\varphi_c) < \varepsilon) \geq -I(\varphi_c,c,T) \geq -I(\varphi_c). 
\]
Hence we are done.  
\end{proof}

Since the large deviation upper bound with the same rate function $I$ was established in our previous work as mentioned in Section \ref{unitary_process}, we arrive at the following theorem:

\begin{theorem}\label{T5.6} 
For any $\varphi_c \in TS^\omega(C^*_R\langle x,u\rangle)$ with $c=c^* \in \mathbb{C}\langle x,u\rangle$, 
\[
\lim_{\varepsilon\searrow0} \begin{Bmatrix} \displaystyle \varlimsup_{N\to\infty} \\ \displaystyle \varliminf_{N\to\infty} \end{Bmatrix} \frac{1}{N^2}\log\mathbb{P}(d(\varphi_N,\varphi_c) < \varepsilon) = -\frac{1}{2}\int_0^\infty \sum_{i=1}^n \Vert E_t^{\varphi_c^t}(\mathfrak{D}_{t,i}c)\Vert_{L^2(\varphi_c)}^2\,dt
\]
holds. 
\end{theorem}

\section{The (matrix) liberation process}\label{matrix_liberation_process} 

The original objective of this series of works is to investigate the matrix liberation process introduced in the first paper \cite{Ueda:JOTP19}. In this section, we apply the results obtained so far to the study of this process. To this end, we slightly modify our notation. The maps $\delta_{t,i}$, $\theta$, $\mathfrak{D}_{t,i}$, $\Pi^t$, and the rate function $I$ appearing earlier in this paper will be denoted by $^u\delta_{t,i}$, $^u\theta$, $^u\mathfrak{D}_{t,i}$, and $^u I$, respectively. That is, we add the prefix ``$u$" to the upper left of each symbol to designate objects associated with unitary processes.   

\subsection{Review on our previous works}

Let $C^*_R\langle x\rangle \subset C^*_R\langle x,v\rangle$ be the universal unital $C^*$-algebras generated by the indeterminates $x_{i,j}(t) = x_{ij}(t)^*$, $i = 1,\dots,n+1$, $j = 1,\dots$, $t \geq 0$, subject to the norm constraints $\Vert x_{ij}(t)\Vert \leq R$ (for a given constant $R > 0$), and the indeterminates $v_i(t)$, $i=1,\dots,n$, $t \geq 0$ satisfying $v_i(t)^* v_i(t) = 1 = v_i(t)v_i(t)^*$. The universal unital $*$-algebras $\mathbb{C}\langle x \rangle \subset \mathbb{C}\langle x,v\rangle$ generated by the same indeterminates are faithfully embedded into the $C^*$-algebras $C^*_R\langle x\rangle \subset C^*_R\langle x,v\rangle$ in a natural fashion. Note that these notations, and those that follow, differ slightly from our previous papers to streamline the description, as noted earlier. 

For each $i = 1,\dots,n$ and $t \geq 0$, there exists a unique derivation $\delta_{t,i} : \mathbb{C}\langle x\rangle \to \mathbb{C}\langle x,v\rangle\otimes\mathbb{C}\langle x,v\rangle$ defined by:  
\begin{equation*}  
\delta_{t,i} x_{i' j}(t') := \delta_{i,i'} \mathbf{1}_{[0,t']}(t)(x_{ij}(t')v_i(t'-t)\otimes v_i(t'-t)^* - v_i(t'-t)\otimes v_i(t'-t)x_{ij}(t')). 
\end{equation*}
Let $\theta$ be the flip-multiplication map from $\mathbb{C}\langle x,v\rangle\otimes\mathbb{C}\langle x,v\rangle$ to $\mathbb{C}\langle x,v\rangle$, defined by $\theta(a\otimes b) = ba$. We set $\mathfrak{D}_{t,i} := \theta\circ\delta_{t,i} : \mathbb{C}\langle x\rangle \to \mathbb{C}\langle x,v\rangle$.  

A tracial state $\tau$ on $C^*_R\langle x \rangle$ is \emph{continuous} if $(t_1,\dots,t_m) \mapsto \tau(x_{i_1 j_1}(t_1)\cdots x_{i_m,j_m}(t_m))$ is continuous for all $i_k, j_k$ and $m \in \mathbb{N}$. We denote the set of all continuous tracial states by $TS^c(C^*_R\langle x \rangle)$, which forms a complete metric space. See \cite[Section 1, Remark on Part I]{Ueda:CJM21}. Any $\tau \in TS^c(C^*_R\langle x \rangle)$ can be extended to a tracial state $\widetilde{\tau}$ on $C^*_R\langle x,v\rangle$ such that $v(t) = (v_1(t),\dots,v_n(t))$ is an $n$-dimensional, left-increment, free unitary Brownian motion, $*$-freely independent of $C^*_R\langle x\rangle$ under $\widetilde{\tau}$ as in subsection 2.5. 

For each $t \geq 0$, there is a unique unital $*$-homomorphism $\Pi^t : C^*_R\langle x,v\rangle \to C^*_R\langle x,v \rangle$ mapping each $x_{ij}(t')$ to: 
\[
x_{ij}^t(t') := v_i((t'-t)\vee0)x_{ij}(t\wedge t')v_i((t'-t)\vee0)^*, 
\] 
while keeping each $v_i(t)$ invarinat. For a given $\tau \in TS^c(C^*_R\langle x \rangle)$ we define $\tau^t \in TS^c(C^*_R\langle x \rangle)$ to be the restriction of $\widetilde{\tau}\circ\Pi^t$ to $C^*_R\langle x\rangle$. 

Let $\tau \in TS^c(C^*_R\langle x \rangle)$ be fixed for the time being. Let $L^2(\widetilde{\tau})$ be the $L^2$-space associated with $\widetilde{\tau}$, into which there is a canonical linear map from $C^*_R\langle x,v\rangle$ with dense image. For each $t \geq 0$, let $\mathbb{C}\langle x\rangle_t$ be the unital $*$-subalgebra of $\mathbb{C}\langle x,v\rangle$ generated by the variables $x_{ij}(s)$ with $s \leq t$. We denote by $C^*_R\langle x\rangle_t$ and $L^2(\tau)_t$ the closures of the images of $\mathbb{C}\langle x \rangle_t$ in $C^*_R\langle x\rangle$ and $L^2(\widetilde{\tau})$, respectively. As in Subsection 2.3, we also consider the $L^\infty$-spaces  $L^\infty(\tau)_t = L^2(\tau)_t \cap L^\infty(\tau) \subseteq L^\infty(\tau) \subset L^\infty(\widetilde{\tau})$.  

Let $X_N = (X_{N,ij})_{i=1,\dots,n+1, j = 1,\dots}$ be $N\times N$ deterministic self-adjoint matrices with $\Vert X_{N,i,j}\Vert \leq R$ ({\it n.b.}, we used the notation $\Xi(N) = (\Xi_i(N))_{i=1}^{n+1}$ and $\Xi_i(N) = (\xi_{ij}(N))_{j = 1,\dots}$ instead in the previous papers). The \emph{matrix liberation process} $X^\mathrm{lib}_N(t) := (X^\mathrm{lib}_{N,ij}(t))_{i=1,\dots,n+1, j  = 1,\dots}$ starting at $X_N$ is defined by $X^\mathrm{lib}_{N,ij}(t) := U_{N,i}(t)X_{N,ij}U_{N,i}(t)^*$, $i=1,\dots,n$, and $X^\mathrm{lib}_{N,n+1 j}(t) := X_{N,n+1 j}$, where $U_N(t) = (U_{N,1}(t),\dots, U_{N,n}(t))$ is an $n$-dimensional Brownian motion on the unitary group $\mathrm{U}(N)$ of rank $N$. For ease of notation, we set $U_{N,n+1}(t) \equiv I_N$ in what follows; accordingly, all $X^\mathrm{lib}_{N,ij}(t)$ admit a single representation: 
\[
X^\mathrm{lib}_{N,ij}(t) := U_{N,i}(t)X_{N,ij}U_{N,i}(t)^*
\]
for $i=1,\dots,n+1$. This process $X^\mathrm{lib}_N(t)$ induces a continuous tracial state $\tau_N := \tau_{X^\mathrm{lib}_N} \in TS^c(C^*_R\langle x\rangle)$ via the unique unital $*$-homomorphism from $C^*_R\langle x\rangle$ to $M_N(\mathbb{C})$ mapping each generator $x_{ij}(t)$ to $X^\mathrm{lib}_{N,ij}(t)$.  Assuming $X_N = (X_{N,ij})_{i=1,\dots,n+1, j \in \mathbb{N}}$ has a limit joint distribution $\sigma_0$ as $N\to\infty$, we previously established in \cite{Ueda:JOTP19} that the sequence $\mathbb{P}(\tau_N \in \,\cdot\,)$ satisfies a large deviation upper bound with speed $N^2$ and the good rate function $I = I_{\sigma_0}^\mathrm{lib} : TS^c(C^*_R\langle x\rangle) \to [0,+\infty]$ defined as follows.  

Let $\tau \in TS^c(C^*_R\langle x \rangle)$ be arbitrarily fixed for the time being. For each $t \geq 0$, let $E_t^{\widetilde{\tau}}$ denote the orthogonal projection from $L^2(\widetilde{\tau})$ onto $L^2(\tau)_t$, which induces a unique $\widetilde{\tau}$-preserving (normal) conditional expectation from $L^\infty(\widetilde{\tau})$ onto $L^\infty(\tau)_t$. Let $\sigma_0^\mathrm{lib} \in TS^c(C^*_R\langle x\rangle)$ be the distribution of the liberation process $x_{ij}(t) = w_i(t) x_{ij} w_i(t)^*$, where the variables $x_{ij}$ are distributed according to $\sigma_0$ and $w(t) = (w_1(t),\dots,w_n(t))$ form an $n$-dimensional, left-increment, free unitary Brownian motion (with $w_{n+1}(t) :\equiv 1$) that is $*$-freely independent of $\{x_{ij}\}$. With these notations we define:  
\begin{equation*}
I(\tau;a,T) := \tau^T(a) - \sigma_0^{\mathrm{lib}}(a) - \frac{1}{2} \int_0^T \sum_{i=1}^n \Vert E_t^{\widetilde{\tau}}(\Pi^t\mathfrak{D}_{t,i}a)\Vert_{L^2(\widetilde{\tau})}^2\,dt
\end{equation*}
for any $a = a^* \in \mathbb{C}\langle x\rangle$ and $T > 0$. We then define the rate function as:  
\begin{equation*}
I(\tau) := \sup\big\{ I(\tau;a,T);\, T > 0, \ a = a^* \in \mathbb{C}\langle x\rangle\big\}. 
\end{equation*}   
Consequently, the desired rate function is constructed in exactly the same manner as in the setting of unitary Brownian motion. In particular, if $\tau$ does not agree with $\sigma_0$ on the polynomials in $x_{ij}(0)$, then $I(\tau)$ must necessarily be $+\infty$.  

\subsection{Study on the rate function(s)} 

We will investigate the rate function for the matrix liberation process by embedding it into the framework of unitary Brownian motion. 

We first re-index the indeterminates $x=(x_j)_{j\geq 1}$ as $x_{ij}$ with $i=1,\dots,n+1$, $j = 1,\dots$. Let $a \in C^*_R\langle x \rangle \mapsto {^u a} \in C^*_R\langle x,u\rangle$  be a unital $*$-homomorphism mapping each $x_{ij}(t)$ to $^u x_{ij}(t) := u_i(t) x_{ij} u_i(t)^*$, where we set $u_{n+1}(t) :\equiv 1$. This mapping induces a continuous map $\varphi \in TS^c(C^*_R\langle x,u\rangle) \mapsto { _u \varphi} \in TS^c(C^*_R\langle x\rangle)$ vis pull-back, denoted by ${_u \varphi}(a) := \varphi({^u a})$ for every $a \in C^*_R\langle x\rangle$. Notably, the relation $\tau_{X^\mathrm{lib}_N} = {_u \varphi}_{(X_N,U_N)}$ holds. That is, the distribution of the matrix liberation process $X^\mathrm{lib}_N(t)$ is precisely the pull-back of the distribution of the unitary Brownian motion via the map $a \mapsto {^u a}$. This mapping $a \mapsto {^u a}$ uniquely extends to a unital $*$-homomorphism from $C^*_R\langle x,v\rangle$ to $C^*_R\langle x,u,\widetilde{u} \rangle$ by mapping each $v_i(t)$ to $\widetilde{u}_i(t)$. It is then clear that $\widetilde{_u \varphi} = {_u\widetilde{\varphi}}$ holds for any $\varphi \in TS^c(C^*_R\langle x,u\rangle)$. 

\begin{lemma}\label{L6.1}
For every $a \in \mathbb{C}\langle x\rangle$ and each $i=1,\dots,n$, we have ${^u\Pi^t}\circ{^u\mathfrak{D}_{t,i}}{^u a} = -\sqrt{-1}\,{^u\big(\Pi^t\circ\mathfrak{D}_{t,i}a\big)}$. 
\end{lemma}
\begin{proof}
Any $a \in \mathbb{C}\langle x\rangle$ is a linear combination of words, say $w = x_{i_1 j_1}(t_1)\cdots x_{i_m j_m}(t_m)$. Thus, we may assume that $a = w$ by linearity of the maps involved. By tedious calculations (see Lemma \ref{L3.1} too) we have
\begin{equation}\label{Eq6.1}
\begin{aligned}
^u(\Pi^t\circ\mathfrak{D}_{t,i}w) 
&= 
\sum_{\substack{w = w_1 x_{i,j}(t') w_2 \\ t \leq t'}} {^u\big(}[v_i(t'-t)^*\,\Pi^t(w_2 w_1)v_i(t'-t), x_{ij}(t)]\big) \\
&= 
\sum_{\substack{w = w_1 x_{i,j}(t') w_2 \\ t \leq t'}} [\widetilde{u}_i(t'-t)^*\,{^u\Pi^t}({^u w_2} {^u w_1})\widetilde{u}_i(t'-t), {^u x_{ij}}(t)]\big), \\
{^u\Pi^t}\circ{^u\mathfrak{D}_{t,i}}{^u w} 
&=
-\sqrt{-1}\sum_{\substack{w = w_1 x_{i,j}(t') w_2 \\ t \leq t'}} {^u\Pi^t}([(u_i(t')u_i(t)^*)^*\,{^u w_2}{^u w_1}(u_i(t')u_i(t)^*), {^u x_{ij}(t)}]) \\
&= 
-\sqrt{-1} \sum_{\substack{w = w_1 x_{i,j}(t') w_2 \\ t \leq t'}} [(\widetilde{u}_i(t'-t)^*\,{^u\Pi^t}({^u w_2}{^u w_1})\widetilde{u}_i(t'-t), {^u x_{ij}(t)}]. 
\end{aligned}
\end{equation} 
Hence we are done.   
\end{proof} 

The mapping $a \mapsto {^u a}$ induces an isometric linear map from $L^2({_u \widetilde{\varphi}})$ to $L^2(\widetilde{\varphi})$ by the definition of ${_u\varphi}$. This extension is still denoted by $\xi \mapsto {^u\xi}$. 

\begin{lemma}\label{L6.2} 
For every $a \in \mathbb{C}\langle x\rangle$ and every $t \geq 0$, 
\[
{^u\big(E_t^{_u\widetilde{\varphi}}(\Pi^t\circ\mathfrak{D}_{t,i}a)\big)} 
= 
{^u\big(E_\infty^{_u\widetilde{\varphi}}(\Pi^t\circ\mathfrak{D}_{t,i}a)\big)} 
=
E_\infty^{\widetilde{\varphi}}({^u \Pi^t}\circ{^u\mathfrak{D}}_{t,i}{^u a}) 
= 
E_t^{\varphi^t}({^u\mathfrak{D}}_{t,i}{^u a})
\]
holds and sits in the image of the unital $*$-subalgebra of $\mathbb{C}\langle x,u\rangle$ generated by ${^u x}_{ij} = u_i(s)x_{ij}u_i(s)^*$ with $s \leq t$ in $L^\infty(\varphi)$.   
\end{lemma}
\begin{proof}
Since $\Pi^t(\mathbb{C}\langle x\rangle)$ sits in the unital $*$-subalgebra generated by the $x_{ij}(s)$ with $s \leq t$ and the $v_i(t)$ and since those $x_{ij}(s)$ and the $v_i(t)$ are $*$-freely independent, it follows by the idea of \cite[Theorem 19 and Exercises 19,20 in section 2.5]{MingoSpeicher:Book} (see also \cite[Lemmas 4.3, 4.4]{Ueda:JOTP19}) that $E_t^{_u\widetilde{\varphi}}(\Pi^t\circ\mathfrak{D}_{t,i}a) = E_\infty^{_u\widetilde{\varphi}}(\Pi^t\circ\mathfrak{D}_{t,i}a)$ holds and belongs to $\mathbb{C}\langle x\rangle_t$ inside $L^\infty({_u\varphi})$. Note that the image of $\mathbb{C}\langle x\rangle_t$ under the mapping $a \mapsto {^u a}$ is the unital $*$-subalgebra of $\mathbb{C}\langle x,u\rangle$ generated by ${^u x}_{ij}(s)$ with $s \leq t$. Hence it suffices to prove only the second equality thanks Corollary \ref{C3.3}. 

For any $b \in C^*_R\langle x \rangle$ we have 
\begin{align*} 
\left({^u\big(}E_\infty^{_u\widetilde{\varphi}}(\Pi^t\circ\mathfrak{D}_{t,i}a)\big)\,\Big|\,{^u b}\right)_{L^2(\widetilde{\varphi})} 
&=
\left(E_\infty^{_u\widetilde{\varphi}}(\Pi^t\circ\mathfrak{D}_{t,i}a)\,\Big|\,b\right)_{L^2(_u\widetilde{\varphi})} \\
&= 
\left(\Pi^t\circ\mathfrak{D}_{t,i}a\,\Big|\,b\right)_{L^2(_u\widetilde{\varphi})} \\
&= 
{_u\widetilde{\varphi}}((\pi^t\circ\mathfrak{D}_{t,i}a)\,b^*) \\
&= 
\widetilde{\varphi}(({^u\Pi^t}\circ{^u\mathfrak{D}_{t,i}}{^u a})\,{^u b}^*) \\
&=
\left({^u\Pi^t}\circ{^u\mathfrak{D}_{t,i}}{^u a}\,\Big|\,{^u b}\right)_{L^2(\widetilde{\varphi})} \\
&= 
\left(E_\infty^{\widetilde{\varphi}}\big({^u\Pi^t}\circ{^u\mathfrak{D}_{t,i}}{^u a}\big)\,\Big|\,{^u b}\right)_{L^2(\widetilde{\varphi})}, 
\end{align*} 
where the last equality follows from ${^u b}$ is in $C^*_R\langle x,u\rangle$. This immediately implies the desired equality. 
\end{proof} 

By combining the above lemma with the trivial identity 
\begin{equation}\label{Eq6.2}
{_u\varphi}^t(a) = {_u\widetilde{\varphi}}(\Pi^t(a)) = \widetilde{\varphi}({^u\Pi^t}({^u a})) = \varphi^t({^u a}), \qquad a \in C^*_R\langle x \rangle, 
\end{equation}
we observe that 
\begin{equation}\label{Eq6.3}
I({_u \varphi};a,T) = {^u I}(\varphi;{^u a},T) 
\end{equation}
holds for every $\varphi \in TS^c(C^*_R\langle x,u\rangle)$, $a = a^* \in \mathbb{C}\langle x\rangle$ and $T > 0$. Consequently, we conclude that 
\begin{equation}\label{Eq6.4}
I(\tau) \leq \inf\{ {^u I}(\varphi); \varphi \in TS^c(C^*_R\langle x,u\rangle)\ \text{with ${_u\varphi} = \tau$}\}
\end{equation}
for every $\tau \in TS^c(C^*_R\langle x\rangle)$, where the right-hand side is defined to be $+\infty$ if no such $\varphi$ exists.  

We then establish a counterpart of \cite[Theorem 5.2(2)]{BianeCapitaineGuionnet:InventMath03} for the liberation process, analogous to Theorem \ref{T3.12}. To this end, we formulate a suitable Hilbert space corresponding to $\mathcal{H}(\varphi)$ from Section 3 and  prove a counterpart to Lemma \ref{L3.10} within this framework. 

As in Section 3, let $\mathcal{H}(\tau)$ be the closure (in $L^2([0,+\infty);L^2(\tau)^n)$) of the space of all bounded, compactly supported, piecewise continuous functions $\xi : [0,+\infty) \mapsto L^2(\tau)^n$ such that $\xi(t) = (\xi_1(t),\dots,\xi_n(t)) \in (L^2(\tau)_t)^n$ for each $t \geq 0$. It is clear that Lemma \ref{L3.8} remains valid in this setting, and we use it hereafter without further comment. We then define the closed subspace $\mathcal{H}^\mathrm{lib}(\tau)$ as the closure of the linear span of functions of the form: 
\[
t \mapsto ([\xi_1(t),x_{1 j_1}(t)],\dots,[\xi_n(t),x_{n j_n}(t)]), 
\]
where $\xi(t) = (\xi_1(t),\dots,\xi_n(t)) \in \mathcal{H}(\tau)_T = \mathbf{1}_{[0,T]}\cdot\mathcal{H}(\tau)$ for some $T > 0$. By Lemma \ref{L6.1} and the identities \eqref{Eq6.1}, the functions 
\begin{equation}\label{Eq6.5}
t \mapsto (-\sqrt{-1}E_t^{\widetilde{\tau}}(\Pi^t\circ\mathfrak{D}_{t,1}a),\dots,-\sqrt{-1}E_t^{\widetilde{\tau}}(\Pi^t\circ\mathfrak{D}_{t,n}a))
\end{equation}
 for $a = a^* \in \mathbb{C}\langle x \rangle$ belong to $\mathcal{H}^\mathrm{lib}(\tau)^\mathrm{sa} = \{ \xi = \xi^* \in \mathcal{H}^\mathrm{lib}(\tau)\}$. However, verifying that these elements form a dense linear submanifold of $\mathcal{H}^\mathrm{lib}(\tau)^\mathrm{sa}$ requires a argument similar to Lemma \ref{L3.10}, the key to which is Lemma \ref{L3.8}. Below, we provide the proof of the counterpart of Lemma \ref{L3.8} in this context, as the original proof was left to the reader.  
 
 \begin{lemma}\label{L6.3} For every $s > r \geq 0$, every $a \in \mathbb{C}\langle x\rangle_r$ and all $i,i' = 1,\dots,n$, the following identity holds:  
 \[
 E_\infty^{\widetilde{\tau}}(\Pi^t\circ\mathfrak{D}_{t,i}((e^s \mathring{x}_{i' j}(s) - e^r \mathring{x}_{i' j}(r))a)) = \delta_{i,i'} \mathbf{1}_{(r,s]}(t)\,[a, e^t x_{ij}(t)], \qquad t \geq 0, 
 \]
 with setting $\mathring{x}_{ij}(t) := x_{ij}(t) - \tau(x_{ij}(t))1$.  
\end{lemma} 
\begin{proof}
By linearity we may assume that $a = w$ as in the proof of Lemma \ref{L6.1}. Let us compute $E_\infty^{\widetilde{\tau}}(\Pi^t\circ\mathfrak{D}_{t,i}(e^s\mathring{x}_{i' j}(s) w))$ for every $s \geq r$. 

By the calculation \eqref{Eq6.1} we have 
\begin{equation}\label{Eq6.6}
\begin{aligned}
E_\infty^{\widetilde{\tau}}(\Pi^t\circ\mathfrak{D}_{t,i}(e^s\mathring{x}_{i' j}(s) w)) 
&= 
\delta_{i,i'} \mathbf{1}_{[0,s]}(t)\,e^s[E_\infty^{\widetilde{\tau}}(v_i(s-t)^* \Pi^t(w) v_i(s-t)),x_{ij}(t)] \\
&\quad + \sum_{\substack{ w=w_1 x_{ij'}(t')w_2 \\ t \leq t' (\leq r)}} e^s[E_\infty^{\widetilde{\tau}}(v_i(t' - t)^* \Pi^t(w_2 \mathring{x}_{i' j}(s) w_1) v_i(t'-t)), x_{ij'}(t)],   
\end{aligned}
\end{equation}
where we used $[\,\cdot\,,\mathring{x}_{ij}(t)] = [\,\cdot\,,x_{ij}(t)]$. 

When $r \leq t \leq s$, we have $\Pi^t(w) = w$ and hence, for any $b \in \mathbb{C}\langle x \rangle$,  
\begin{align*}
&(E_\infty^{\widetilde{\tau}}(v_i(s-t)^* \Pi^t(w) v_i(s-t))\,|\,b)_{L^2(\widetilde{\tau})} \\
&= 
(v_i(s-t)^* w v_i(s-t)\,|\,b)_{L^2(\widetilde{\tau})} \\
&=  
\widetilde{\tau}(v_i(s-t)^* w v_i(s-t)b^*) \\
&= 
\widetilde{\tau}(w)\widetilde{\tau}(b^*) + |\widetilde{\tau}(v_i(s-t))|^2 \widetilde{\tau}(wb^*) - |\widetilde{\tau}(v_i(s-t))|^2 \widetilde{\tau}(w)\widetilde{\tau}(b^*) \\
&= 
\widetilde{\tau}((\widetilde{\tau}(w)1 + e^{-(s-t)}\mathring{w})b^* \\
&=
((1-e^{-(s-t)})\widetilde{\tau}(w)1 + e^{-(s-t)}w\,|\,b)_{L^2(\widetilde{\tau})}, 
\end{align*}  
where we used the free independence of $v_i(s-t)$ against the others, and $\widetilde{\tau}(v_i(s-t)) = e^{-(s-t)/2}$ (see \cite[Lemma 1]{Biane:Fields97}). Thus, $E_\infty^{\widetilde{\tau}}(v_i(s-t)^* \Pi^t(w) v_i(s-t)) = (1-e^{-(s-t)})\tau(w)1 + e^{-(s-t)}w$ holds in this case. When $t < r$, for any $b \in \mathbb{C}\langle x \rangle$ we similarly have 
 \begin{align*} 
&(E_\infty^{\widetilde{\tau}}(v_i(s-t)^* \Pi^t(w) v_i(s-t))\,|\,b)_{L^2(\widetilde{\tau})} \\
&= 
e^{-(s-r)}\widetilde{\tau}(v_i(r-t)^*\Pi^t(w)v_i(r-t) b^*) + \widetilde{\tau}(\Pi^t(w))\widetilde{\tau}(b^*) - e^{-(s-r)}\widetilde{\tau}(\Pi^t(w))\widetilde{\tau}(b^*) \\
&= 
((1-e^{-(s-r)})\widetilde{\tau}(\Pi^t(w))1 + e^{-(s-r)} E_\infty^{\widetilde{\tau}}(v_i(r-t)^*\Pi^t(w)v_i(r-t))\,|\,b)_{L^2(\widetilde{\tau})},   
\end{align*}
where we used the left free-increment property of $t \mapsto v_i(t)$ under $\widetilde{\tau}$ together with trace property as just before. Hence we have obtained that $E_\infty^{\widetilde{\tau}}(v_i(s-t)^* \Pi^t(w) v_i(s-t)) = (1-e^{-(s-r)})\widetilde{\tau}(\Pi^t(w))1 + e^{-(s-r)} E_\infty^{\widetilde{\tau}}(v_i(r-t)^*\Pi^t(w)v_i(r-t))$. Therefore, the first term of the right-hand side of the identity \eqref{Eq6.6} becomes 
\begin{equation}\label{Eq6.7}
\delta_{i,i'}\mathbf{1}_{[0,r)}(t)[e^r E_\infty^{\widetilde{\tau}}(v_i(r-t)^*\Pi^t(w)v_i(r-t)),x_{ij}(t)] + 
\delta_{i,i'}\mathbf{1}_{[r,s]}(t)[w,e^t x_{ij}(t)].  
\end{equation}

While paying attention to $t \leq t' \leq r \leq s$, for any $b \in \mathbb{C}\langle x\rangle$ we compute 
\begin{align*}
&(E_\infty^{\widetilde{\tau}}(v_i(t' - t)^* \Pi^t(w_2 \mathring{x}_{i' j}(s) w_1) v_i(t'-t))\,|\,b)_{L^2(\widetilde{\tau})} \\
&= 
\widetilde{\tau}(v_i(t' - t)^* \Pi^t(w_2) (v_i(s-t)v_i(r-t)^*)\Pi^t(\mathring{x}_{i' j}(r))(v_i(s-t)v_i(r-t)^*)^* \Pi^t(w_1) v_i(t'-t) b^*) \\
&= 
(e^{-(s-r)} E_\infty^{\widetilde{\tau}}(v_i(t' - t)^* \Pi^t(w_2 \mathring{x}_{i' j}(r) w_1) v_i(t'-t))\,|\,b)_{L^2(\widetilde{\tau})} 
\end{align*}   
as before. Hence we have obtained that 
\[
E_\infty^{\widetilde{\tau}}(v_i(t' - t)^* \Pi^t(w_2 \mathring{x}_{i' j}(s) w_1) v_i(t'-t)) = e^{-(s-r)} E_\infty^{\widetilde{\tau}}(v_i(t' - t)^* \Pi^t(w_2 \mathring{x}_{i' j}(r) w_1) v_i(t'-t)), 
\]
and hence the second term of the right-hand side of equation \eqref{Eq6.6} becomes 
\begin{equation}\label{Eq6.8}
\sum_{\substack{w=w_1 x_{ij'}(t')w_2 \\ t \leqq t' (\leqq r)}} [E_\infty^{\widetilde{\tau}}(v_i(t' - t)^* \Pi^t(w_2 (e^r \mathring{x}_{i' j}(r)) w_1) v_i(t'-t)),x_{ij'}(t)]
\end{equation}
Since both the first term of \eqref{Eq6.7} and \eqref{Eq6.8} are independent of $s$ within $s \geq r$, we have 
\[
E_\infty^{\widetilde{\tau}}(\Pi^t\circ\mathfrak{D}_{t,i}(e^s\mathring{x}_{i' j}(s) w)) - E_\infty^{\widetilde{\tau}}(\Pi^t\circ\mathfrak{D}_{t,i}(e^r\mathring{x}_{i' j}(r) w)) = \delta_{i,i'}\mathbf{1}_{(r,s]}(t)[w,e^t x_{ij}(t)]. 
\]
Hence we are done. 
\end{proof}

With this lemma, the essentially same discussion as the proof of Lemma \ref{L3.10} shows that all the functions of the form \eqref{Eq6.5} form a dense linear submanifold of $\mathcal{H}(\tau)$. Then, with \cite[Lemma 5.3]{Ueda:JOTP19} we can prove the next counterpart of Theorem \ref{T3.12} in the same way as there. 

\begin{theorem}\label{T6.4} Assume that $I(\tau) < +\infty$. Then there exists a unique (up to almost everywhere equivalence) self-adjoint function $\xi : [0,+\infty) \to L^2(\tau)^n$ that belongs to $\mathcal{H}^\mathrm{lib}(\tau)$ such that the following properties hold: 
\begin{itemize}
\item[(i)] $\displaystyle I(\varphi) = \frac{1}{2} \Vert\xi\Vert_{L^2}^2$, where $\Vert\,\cdot\,\Vert_{L^2}$ denotes the $L^2$-norm for functions over $[0,+\infty)$ taking its values in $L^2(\tau)^n$.   
\item[(ii)] For every $T > 0$ and any $a \in \mathbb{C}\langle x\rangle$, the following identity holds: 
\[
\tau^T(a) = \sigma_0^\mathrm{lib}(a) + \int_0^T \sum_{i=1}^n(E_t^{\widetilde{\tau}}(\Pi^t\circ\mathfrak{D}_{t,i}a)\,|\,\xi_i(t))_{L^2(\tau)}\,dt.
\]
\end{itemize}

Moreover, if there exists a square-integrable, self-adjoint function $\xi : [0,+\infty) \to L^2(\tau)^n$ that belongs to $\mathcal{H}^\mathrm{lib}(\tau)$ and satisfies the integral expression in item (ii) above, then item (i) above necessarily holds with this $\xi$.   
\end{theorem} 

\subsection{A weak large deviation result} 
Choose a $c = c^* \in \mathbb{C}\langle x\rangle$, and consider the self-adjoint element ${^u c} \in \mathbb{C}\langle x,u\rangle$ and the smooth tracial state $\varphi_{{^u(c)}} \in TS^\omega(C^*_T\langle x,u\rangle)$, which we write $\varphi_c$ instead only here for the ease of notation. Set $\tau_c := {_u(\varphi_c)} \in TS^c(C^*_R\langle x \rangle)$. By the identity \eqref{Eq6.2} and Lemma \ref{L6.2} together with Lemma \ref{L3.4} we observe that 
\begin{align*}
\tau_c^T(a) 
= 
\varphi_c^T({^u a}) 
&= 
\sigma_0^\mathrm{frBM}({^u a}) + \int_0^T \sum_{i=1}^n \varphi_c^t(E_t^{\varphi_c^t}({^u\mathfrak{D}}_{t,i} {^u a})E_t^{\varphi_c^t}({^u \mathfrak{D}}_{t,i}{^u c}))\,dt \\ 
&= 
\sigma_0^\mathrm{lib}(a) + \int_0^T \sum_{i=1}^n \tau^t(E_t^{\widetilde{\tau}}(\Pi^t\circ\mathfrak{D}_{t,i} {a})E_t^{\widetilde{\tau}}(\Pi^t\circ\mathfrak{D}_{t,i}{c}))\,dt 
\end{align*}  
holds for every $T > 0$ and any $a \in \mathbb{C}\langle x\rangle$. Therefore, we conclude that 
\begin{equation*} 
I(\tau_c) = \frac{1}{2}\int_0^\infty \sum_{i=1}^n \Vert E_t^{\widetilde{\tau}_c}(\Pi^t\circ\mathfrak{D}_{t,i} c)\Vert_{L^2(\tau_c)}^2\,dt = \frac{1}{2}\int_0^\infty \sum_{i=1}^n \Vert E_t^{\varphi_c^t}({^u\mathfrak{D}}_{t,i} {^u c})\Vert_{L^2(\varphi_c)}^2\,dt = {^u I}(\varphi_c). 
\end{equation*}
With this observation, Theorem \ref{P5.5} and \cite[Theorem 5.8]{Ueda:JOTP19} (or its proof) we obtain the following: 

\begin{corollary}\label{C6.5} 
For any $\tau_c \in TS^c(C^*_R\langle x\rangle)$ with $c=c^* \in \mathbb{C}\langle x \rangle$, 
\[
\lim_{\varepsilon\searrow0} \begin{Bmatrix} \displaystyle \varlimsup_{N\to\infty} \\ \displaystyle \varliminf_{N\to\infty} \end{Bmatrix} \frac{1}{N^2}\log\mathbb{P}(d(\tau_N,\tau_c) < \varepsilon) = -I(\tau_c)
\]
holds. 
\end{corollary}
\begin{proof} 
The upper bound was established for all $\tau \in TS^c(C^*_R\langle x\rangle)$ in \cite[Theorem 5.8]{Ueda:JOTP19}. 
 
Since $\varphi \mapsto {_u\varphi}$ is continuous and ${_u(\varphi_N)} = \tau_N$, ${_u(\varphi_c)} = \tau_c$, we observe that
\[
\varliminf_{N\to\infty}\frac{1}{N^2}\log\mathbb{P}(d(\tau_N,\tau_c)<\varepsilon) 
\geq \lim_{\delta\searrow0}\varliminf_{N\to\infty}\log\mathbb{P}(d(\varphi_N,\varphi_c)<\delta) \geq -{^u I}(\varphi_c) = -I(\tau_c)
\]
holds for every $\varepsilon>0$ thanks to Theorem \ref{P5.5}.  
\end{proof} 

\begin{remark}\label{R6.6} {\rm 
The preceding discussion leads to the following observation: Fix $\varphi \in TS^c(C^*_R\langle x,u\rangle)$ such that $I(\varphi) < +\infty$, and let $\xi \in \mathcal{H}(\varphi)$ be as in Theorem \ref{T3.12}. The mapping $\eta  \mapsto {^u\eta}$ from $L^2([0,+\infty); L^2({_u\varphi}))$ to $L^2([0,+\infty);L^2(\varphi))$, defined by $({^u \eta})(t) := {^u (\eta(t))}$, $t \geq 0$, is an isometry. It is not difficult to verify that the image of $\mathcal{H}^\mathrm{lib}({_u\varphi})$ under this map is contained in $\mathcal{H}(\varphi)$. Consequently, there exists a unique $\xi^\mathrm{lib} \in \mathcal{H}^\mathrm{lib}({_u\varphi})$ such that ${^u(\xi^\mathrm{lib})}$ is the orthogonal projection of $\xi$ to the image of $\mathcal{H}^\mathrm{lib}({_u\varphi})$. For any $a \in \mathbb{C}\langle x \rangle$, we have 
\begin{align*} 
({_u\varphi})^T(a) 
= 
\varphi^T({^u a}) 
&= 
\sigma_0^\mathrm{frBM}({^u a}) + (t\mapsto(\mathbf{1}_{[0,T)}E_t^{\varphi^t}({^u\mathfrak{D}}_{t,i}{^u a}))_{i=1}^n\,|\, \mathbf{1}_{[0,T)}\xi)_{L^2} \\
&= 
\sigma_0^\mathrm{frBM}({^u a}) + (t\mapsto(\mathbf{1}_{[0,T)}{^u(E_t^{{_u\widetilde{\varphi}}}(\Pi^t\circ\mathfrak{D}_{t,i}a))})_{i=1}^n\,|\, \mathbf{1}_{[0,T)}\xi)_{L^2} 
\qquad \text{(by Lemma \ref{L6.2})} \\
&= 
\sigma_0^\mathrm{frBM}({^u a}) + (t\mapsto(\mathbf{1}_{[0,T)}{^u(E_t^{{_u\widetilde{\varphi}}}(\Pi^t\circ\mathfrak{D}_{t,i}a))})_{i=1}^n\,|\, \mathbf{1}_{[0,T)}{^u(\xi^\mathrm{lib})})_{L^2} \\
&= 
\sigma_0^\mathrm{lib}(a) + (t\mapsto(\mathbf{1}_{[0,T)} E_t^{{_u\widetilde{\varphi}}}(\Pi^t\circ\mathfrak{D}_{t,i}a))_{i=1}^n\,|\, \mathbf{1}_{[0,T)}\xi^\mathrm{lib})_{L^2} \\
&= 
\sigma_0^\mathrm{lib}(a) 
+ \int_0^T \sum_{i=1}^n {_u\varphi}(E_t^{{_u\widetilde{\varphi}}}(\Pi^t\circ\mathfrak{D}_{t,i}a)\xi^\mathrm{lib}(t))\,dt 
\end{align*}   
Hence, by Theorem \ref{T6.4} we conclude that 
\[ 
I({_u\varphi}) = \frac{1}{2}\Vert\xi^\mathrm{lib}\Vert_{L^2}^2 \leq \frac{1}{2}\Vert\xi\Vert_{L^2}^2 = {^u  I}(\varphi),
\]
where the equality holds if and only if $\xi$ belongs to the image of $\mathcal{H}^\mathrm{lib}({_u\varphi})$ in $\mathcal{H}(\varphi)$.}
\end{remark}  

\appendix 

\section{Non-commutative $L^2$-spaces and affiliated operators} \label{noncomm_L^2}

Let $A$ be a unital $C^*$-algebra and $\tau$ be a tracial state on it. We have the $L^2$-space $L^2(\tau)$, which is the separation/completion of $A$ with respect to the pre-inner product $(a,b) \mapsto \tau(ab^*)$ ($= \tau(b^* a)$).  Following standard analytical practice, we identify $a \in A$ in $L^2(\tau)$ with its canonical image in $L^2(\tau)$. 

Owing to the trace property, the adjoint operation $a \mapsto a^*$ on $A$ extends to $L^2(\tau)$ as an anti-unitary operator. We denote the extension by $\xi \mapsto \xi^*$ for $\xi \in L^2(\tau)$. It is easily verified that $(\xi|\eta)_{L^2(\tau)} = (\eta^*|\xi^*)_{L^2(\tau)}$ holds for every pair $\xi,\eta \in L^2(\tau)$.  

Furthermore, the trace property ensures that left and right multiplications of $A$ on itself extends continuously to $L^2(\tau)$. This endows $L^2(\tau)$ with an $A$-$A$ bimodule structure, $(a,\xi,b) \in A \times L^2(\tau) \times A \mapsto a\xi b \in L^2(\tau)$. One can confirm the identity $(a\xi b)^* = b^*\xi^* a^*$. The images of $A$ under these left and the right actions are denoted by $\mathcal{L}(A)$ and $\mathcal{R}(A)$, respectively. 

We define the $L^\infty$-norm of $\xi \in L^2(\tau)$:  
\[
\Vert \xi\Vert_{L^\infty(\tau)} := \sup\{\Vert\xi a\Vert_{L^2(\tau)}; a \in A,\ \Vert a\Vert_{L^2(\tau)} = \tau(aa^*)^{1/2} \leq 1\} \in [0,+\infty], 
\]  
and the $L^\infty$-space as the subspace of elements with finite $L^\infty$-norm. Moreover, $L^\infty(\tau)$ is closed under the adjoint operation as follows. For any $\xi \in L^\infty(\tau)$ and $a \in A$, 
\begin{align*}
\Vert \xi^* a \Vert_{L^2(\tau)} 
&= 
\sup\{ |(\xi^* a|b)_{L^2(\tau)}|; b \in A,\ \Vert b\Vert_{L^2(\tau)} \leq 1\} \\
&= 
\sup\{ |(a|\xi b)_{L^2(\tau)}|; b \in A,\ \Vert b\Vert_{L^2(\tau)} \leq 1\} \\
&\leq 
\Vert a\Vert_{L^2(\tau)} \Vert \xi \Vert_{L^\infty(\tau)},
\end{align*}   
where the last inequality follows from the Cauchy--Schwarz inequality. This implies $\xi^* \in L^\infty(\tau)$ and $\Vert \xi^*\Vert_{L^\infty(\tau)} \leq \Vert\xi\Vert_{L^\infty(\tau)}$. By symmetry, $\Vert\xi^*\Vert_{L^\infty(\tau)} = \Vert\xi\Vert_{L^\infty(\tau)}$, confirming that the adjoint is an isometry on $L^\infty(\tau)$ as well. 

By construction, each $\xi \in L^\infty(\tau)$ defines a unique bounded operator $L_\xi$ on $L^2(\tau)$ such that $L_\xi a = \xi a$ for all $a \in A$. The operator norm $\Vert L_\xi\Vert$ coincides with $\Vert\xi\Vert_{L^\infty(\tau)}$. 

The next lemma is standard, and its detailed proof is left to the reader. 

\begin{lemma}\label{L-A.1}
The following hold ture: 
\begin{itemize}
\item[(1)] $(L_\xi)^* = L_{\xi^*}$ and $(L_\xi \eta)^* = (L_\eta)^*\xi^*$ for any $\xi,\eta \in L^\infty(\tau)$.  
\item[(2)] For any pair $\xi,\eta \in L^\infty(\tau)$, the $\xi\eta := L_\xi\eta$ falls in $L^\infty(\tau)$ and $\Vert\xi\eta\Vert_{L^\infty(\tau)} \leq \Vert\xi\Vert_{L^\infty(\tau)}\Vert\eta\Vert_{L^\infty(\tau)}$ holds. 
\item[(3)] $(\xi\eta)\zeta = \xi(\eta\zeta)$ and $(\xi\eta)^*=\eta^*\xi^*$ hold for any $\xi,\eta,\zeta\in L^\infty(\tau)$. In particular, the operation $(\xi,\eta)\in L^\infty(\tau)^2\mapsto \xi\eta\in L^\infty(\tau)$ makes $L^\infty(\tau)$ a $*$-algebra. 
\item[(4)] $L^\infty(\tau)$ is a unital $C^*$-algebra with the $L^\infty$-norm.
\item[(5)] $\xi \in L^\infty(\tau) \mapsto L_\xi \in B(L^2(\tau))$ is an injective unital $*$-representation, and $L_a \xi = a\xi$ holds for any $a \in A$ and $\xi \in L^2(\tau)$, where the $a$ of $L_a$ is regarded as an element of $A$ in $L^2(\tau)$. 
\end{itemize}
\end{lemma}

By the preceding lemma, each $\xi \in L^\infty(\tau)$ also defines a unique bounded operator $R_\xi$ on $L^2(\tau)$ such that $R_\xi \eta = \eta\xi$ for all $\eta \in L^\infty(\tau)$. Consistent with our previous notation, the map $\xi \mapsto R_\xi$ represents the extension of the right action of $A$ on $L^2(\tau)$. 

We denote the families of these operators by $\mathcal{L}(L^\infty(\tau)) := \{ L_\xi; \xi \in L^\infty(\tau)\}$ and $\mathcal{R}(L^\infty(\tau)) := \{ R_\xi; \xi \in L^\infty(\tau)\}$, which naturally contain $\mathcal{L}(A)$ and $\mathcal{R}(A)$, respectively. 

\begin{proposition}\label{P-A.2} 
The following identities hold: 
\[
\mathcal{L}(A)'' =  \mathcal{L}(L^\infty(\tau)) = \mathcal{R}(L^\infty(\tau))', \qquad \mathcal{R}(A)'' = \mathcal{R}(L^\infty(\tau)) = \mathcal{L}(L^\infty(\tau))',
\] 
where the prime ($\prime$) denotes the commutant in $B(L^2(\tau))$. In particular, both $\mathcal{L}(L^\infty(\tau))$ and $\mathcal{R}(L^\infty(\tau))$ are von Neumann algebras. Furthermore, $\tau$ extends uniquely to these algebras as a faithful normal tracial state, which we continue to denote by the same symbol $\tau$.
\end{proposition}
\begin{proof} 
For any $\xi,\eta,\zeta \in L^\infty(\tau)$ we have 
$L_\xi R_\eta = R_\eta L_\xi$. It follows that $\mathcal{L}(L^\infty(\tau)) \subseteq \mathcal{R}(L^\infty(\tau))'$ and $\mathcal{R}(L^\infty(\tau)) \subseteq \mathcal{L}(L^\infty(\tau))'$.   

To see the reverse inclusion, choose an $X' \in \mathcal{L}(A)'$ and set $\xi := X'1 \in L^2(\tau)$. Then, for all $a \in A$ we have $a\xi = a\xi  = L_a X'1 = X' L_a 1 = X' a$ in $L^2(\tau)$. This implies that $\xi \in L^\infty(\tau)$ and $R_\xi a = X'a$ for all $a \in A$. Hence $X' = R_\xi \in \mathcal{R}(L^\infty(\tau))$; hence $\mathcal{L}(A)' \subseteq \mathcal{R}(L^\infty(\tau))$. In the same way, we obtain $\mathcal{R}(A)' \subset \mathcal{L}(L^\infty(\tau))$. 

Consequently, $\mathcal{L}(L^\infty(\tau))' \subseteq \mathcal{L}(A)' \subseteq \mathcal{R}(L^\infty(\tau)) \subseteq \mathcal{L}(L^\infty(\tau))'$, and hence $\mathcal{L}(L^\infty(\tau))' = \mathcal{L}(A)' = \mathcal{R}(L^\infty(\tau))$. Similarly, $\mathcal{R}(L^\infty(\tau))' = \mathcal{R}(A)' = \mathcal{L}(L^\infty(\tau))$.    

Furthermore, for any $\xi, \eta \in L^\infty(\tau)$ the tracial property holds:  
\[
(L_\xi L_\eta 1|1)_{L^2(\tau)} = (\eta|\xi^*)_{L^2(\tau)} = (\xi|\eta^*)_{L^2(\tau)} = (L_\eta L_\xi 1|1)_{L^2(\tau)}. 
\]    
Hence, the vector state $(\,\cdot\,1|1)_{L^2(\tau)}$ gives a tracial state on $\mathcal{L}(L^\infty(\tau))$, and similarly on $\mathcal{R}(L^\infty(\tau))$. Since $1 \in A$ is clearly cyclic for both algebras, this tracial vector state is faithful. Its normality is trivial. Finally, since $(L_a 1|1)_{L^2(\tau)} = (R_a 1|1)_{L^2(\tau)} = \tau(a)$ for every $a \in A$, this tracial vector state is a normal extension of $\tau$. 
\end{proof} 

Since the mappings $\xi \mapsto L_\xi$ and $\eta \mapsto  R_\eta$ commute, they induce  an $L^\infty(\tau)$-bimodule structure on $L^2(\tau)$ given by $\xi\zeta\eta := L_\xi R_\eta \zeta = R_\eta L_\xi \zeta$ for all $\xi,\eta \in L^\infty(\tau)$ and $\zeta \in L^2(\tau)$. Furthermore, $\xi \mapsto L_\xi$ is an injective $*$-homomorphism from $L^\infty(\tau)$ onto $\mathcal{L}(L^\infty(\tau))$; thus, $L^\infty(\tau)$ is a $W^*$-algebra. On this algebra, the linear functional $\xi \mapsto (\xi|1)_{L^2(\tau)}$ defines a tracial state that extends the original $\tau$. Hereafter, we continue to denote this tracial state on $L^\infty(\tau)$ by the same symbol $\tau$.    

\medskip
For a general element $\xi \in L^2(\tau)$, we consider the densely defined operator $L^0_\xi : A \subset L^2(\tau) \to L^2(\tau)$ defined by $L^0_\xi a := \xi a$ for every $a \in A$. Using the trace property, we observe that for any $a,b \in A$:   
\[
(L^0_\xi a|b)_{L^2(\tau)} = (\xi a|b)_{L^2(\tau)} = (b^*|a^*\xi^*)_{L^2(\tau)} = (ab^*|\xi^*)_{L^2(\tau)} = (a|\xi^* b)_{L^2(\tau)} = (a|L^0_{\xi^*}b)_{L^2(\tau)}.
\]
This calculation shows that $L^0_\xi$ is closable, and we denote its closure by $L_\xi$. Here is an important fact, which was used in the proof of Lemma \ref{L4.10}. 

\begin{proposition}\label{P-A.3} 
(1) For every $\xi \in L^2(\tau)$ the operator $L_\xi$ is affiliated with the von Neumann algebra $\mathcal{L}(L^\infty(\tau))$. 
 
(2) For $\xi,\eta \in L^2(\tau)$ and $\alpha,\beta \in \mathbb{C}$, we have:  
\[
\overline{\alpha L_\xi + \beta L_\eta} = L_{\alpha\xi + \beta\eta}, \qquad (L_\xi)^* = L_{\xi^*}. 
\]
In particular, if $\xi = \xi^*$, then $L_\xi$ is self-adjoint and admits a spectral decomposition.   
\end{proposition}
\begin{proof}
(1) One has to confirm that $u' L_\xi u'{}^* = L_\xi$ for every $\xi \in L^2(\tau)$ and every unitary operator $u' \in \mathcal{L}(L^\infty(\tau))'$. See \cite[section 9.7]{StratilaZsido:Book}. By a simple consideration, it suffices to prove that $Y' L^0_\xi \subseteq L_\xi Y'$ holds for every $\xi \in L^2(\tau)$ and every $Y' \in \mathcal{L}(L^\infty(\tau))'$. 

By Proposition \ref{P-A.2} with the von Neumann Double Commutatnt Theorem, there exists a sequence $y_n \in A$ such that $R_{y_n} \to Y'$ in the strong operator topology on $L^2(\tau)$. For each $a \in A$ we observe that $ay_n \to R_{y_n}a = Y' a$ and $L_\xi ay_n = \xi ay_n = R_{y_n} \xi a \to Y' \xi a = Y' L^0_\xi a$ as $n\to\infty$. Since $L_\xi$ is closed, it follows that $Y' a$ belongs to the domain of $L_\xi$ and $L_\xi Y' a = Y' L^0_\xi a$. 

(2) This follows directly from standard results on von Neumann algebras \cite[section 9.8 and E.9.26]{StratilaZsido:Book}, as $\mathcal{L}(L^\infty(\tau))$ is finite by the virtue of the tracial state.
\end{proof}

For any $\xi=\xi^* \in L^2(\tau)$ and $z \in \mathbb{C}\setminus\mathbb{R}$, the operator $L_{z1-\xi}=z1-L_\xi$ has a (unique) bounded inverse belonging to $\mathcal{L}(L^\infty(\tau))$. Consequently, there exists a unique $\xi_z \in L^\infty(\tau)$ such that $L_{z1-\xi}^{-1}=L_{\xi_z}$. 

For another $\eta = \eta^* \in L^2(\tau)$, we obtain the resolvent identity: 
\begin{align*}
\xi_z - \eta_z
= 
L_{\xi_z - \eta_z}1 
&=
((z1 - L_\xi)^{-1}-(z1-L_\eta)^{-1})1 \\
&=
((z1 - L_\xi)^{-1}(L_\xi - L_\eta)(z1-L_\eta)^{-1})1 \\
&=
L_{\xi_z}(L_\xi-L_\eta)L_{\eta_z}1 \\
&= 
L_{\xi_z}L_{\xi-\eta}R_{\eta_z}1 \\
&=
L_{\xi_z}R_{\eta_z}L_{\xi-\eta}1 \qquad \text{(by Proposition \ref{P-A.3}(1))} \\
&=
\xi_z(\xi-\eta)\eta_z. 
\end{align*}
Similarly, for another $z' \in \mathbb{C}\setminus\mathbb{R}$: 
\begin{align*}
\xi_z - \xi_{z'} 
&= 
L_{\xi_z-\xi_{z'}}1 \\
&=
((z1-L_\xi)^{-1}-(z'1-L_\xi)^{-1})1 \\
&=
(z'-z)(z1-L_\xi)^{-1}(z'1-L_\xi)^{-1}1 \\
&=
(z'-z)L_{\xi_z}L_{\xi_{z'}}1 
=
L_{(z'-z)\xi_z\xi_{z'}}1 
=
(z'-z)\xi_z\xi_{z'}. 
\end{align*}
By identifying $\xi_z$ with the formal expression $(z1-\xi)^{-1}$, all the calculations in the proof of Lemma \ref{L4.10} are rigorously justified. 

\medskip
Let $B$ be a unital $*$-subalgebra of $A$, and denote by $\tau_B$ the restriction of $\tau$ to $B$. From $\tau_B$, we obtain the $L^2$- and $L^\infty$-spaces, $L^2(\tau_B) \supseteq L^\infty(\tau_B)$, as well as the von Neumann algebra $\mathcal{L}(L^\infty(\tau_B))$ acting on $L^2(\tau_B)$. Let $E_B$ be the orthogonal projection from $L^2(\tau)$ onto $L^2(\tau_B)$, where $L^2(\tau_B)$ is naturally identified with closure of $B$ in $L^2(\tau)$. 

The inclusion $L^\infty(\tau_B) \subseteq L^\infty(\tau)$ is somewhat non-trivial, and can be shown as follows. It is readily observed that $E_B$ belongs to the commutant $\mathcal{L}(B)'$ of the collection $\mathcal{L}(B)$ of all $L_b \in B(L^2(\tau))$ for $b \in B$. Fix $\xi \in L^\infty(\tau_B)$ and assume that the left-multiplication operator $L^B_\xi$ (acting on $L^2(\tau_B)$) is a unitary operator for the time being. By naturally regarding $L^B_\xi$ as an operator on $L^2(\tau)$, we obtain the following identity for any $S_i \in \mathcal{L}(B)'$ and $\eta_i \in L^2(\tau_B) \subseteq L^2(\tau)$: 
\[
\left\Vert \sum_i S_i L_\xi^B \eta_i\right\Vert_{L^2(\tau)}^2 
= 
\sum_{i,j} (E_B S_j^* S_i E_B L^B_\xi \eta_i| L^B_\xi \eta_j)_{L^2(\tau)} 
=   
\left\Vert \sum_i S_i \eta_i\right\Vert_{L^2(\tau)}^2, 
\]
where we use the fact that $E_B S_j^* S_i E_B$ belongs to $E_B \mathcal{L}(B)' E_B = \mathcal{L}(L^\infty(\tau_B))'$. Since $\mathcal{L}(B)' \supseteq \mathcal{L}(A)'$, there exists an isometric operator $V \in B(L^2(\tau))$ such that $V S\eta = S L^B_\xi \eta$ for any $S \in \mathcal{L}(B)'$ and $\eta \in L^2(\tau_B)$. It follows that $V \in \mathcal{L}(B)'' \subseteq \mathcal{L}(A)'' = \mathcal{L}(L^\infty(\tau))$. In particular, there exists an element $\tilde{\xi} \in L^\infty(\tau)$ such that $V = L_{\tilde{\xi}}$. Since $1 \in L^2(\tau_B)$, we conclude that $\tilde{\xi} = V1 = L_\xi^B 1 = \xi$; thus, $\xi$ belongs to $L^\infty(\tau)$. Since any element of a von Neumann algebra is is written as a linear combination of four unitary elements within the same algebra, this conclusion remains valid for a general $\xi \in L^\infty(\tau_B)$. Consequently, we establish the inclusion $L^\infty(\tau_B) \subseteq L^\infty(\tau)$. 

For any $\xi \in L^\infty(\tau)$ and $b \in B$, we observe that 
\[
\Vert (E_B \xi)b\Vert_{L^2(\tau_B)} = \Vert R_b E_B \xi\Vert_{L^2(\tau)} = \Vert E_B R_b  \xi\Vert_{L^2(\tau)} \leq \Vert   \xi b\Vert_{L^2(\tau)} \leq \Vert\xi\Vert_{L^\infty(\tau)} \Vert b \Vert_{L^2(\tau)} = \Vert\xi\Vert_{L^\infty(\tau)} \Vert b \Vert_{L^2(\tau_B)}.  
\]  
This inequality implies that $E_B \xi$ belongs to $L^\infty(\tau_B)$ with the norm estimate $\Vert E_B\xi \Vert_{L^\infty(\tau_B)} \leq \Vert\xi\Vert_{L^\infty(\tau)}$. Consequently, we conclude that $L^\infty(\tau_B) = E_B L^\infty(\tau) = L^2(\tau_B) \cap L^\infty(\tau)$ holds, and that any element of the von Neumann algebra $\mathcal{L}(B)''$ can be represented as $L_\xi$ for some $\xi \in L^\infty(\tau_B)$. Moreover, for any $\xi \in L^\infty(\tau)$ and $\eta \in L^\infty(\tau_B)$ we have $E_B\xi^* = (E_B\xi)^*$, 
$E_B(\xi\eta) = E_B R_\eta \xi = R_\eta E_B\xi = (E_B\xi)\eta$, and 
\[
\tau((E_B\xi)\eta) = ((E_B\xi)\eta|1)_{L^2(\tau)} = (E_B(\xi\eta)|1)_{L^2(\tau)} = (\xi\eta|1)_{L^2(\tau)} = \tau(\xi\eta).
\] 
These identities show that the map $\xi \mapsto E_B\xi$ defines a $\tau$-preserving conditional expectation from $L^\infty(\tau)$ onto $L^\infty(\tau_B)$. The fundamental properties we have provided provide a rigorous justification for all the preceding discussions concerning conditional expectations throughout the main body of this paper. 

\begin{remark}\label{R-A.4}{\rm 
While we have restricted our use of $L^p$-spaces to $p=2,4$ and $\infty$ for technical simplicity, it is well known that the for $p \in [1,2)$ $L^p$-spaces are larger than the $L^2$-space. Thus, it is might be more natural to start with the $L^1$-space $L^1(\tau)$, defined as the separation/completion of $A$ with respect to the $L^1$-seminorm $a \mapsto \tau(|a|)$, and subsequently realize all $L^p(\tau)$ as its subspaces. This construction is straightforward and is left to the interested reader.
}
\end{remark}

}

\end{document}